\documentclass{amsart}
\usepackage{times,mathabx,amsfonts,amssymb,amsrefs,mathrsfs,tikz}
\usepackage{mathabx,amsfonts,amssymb,amsrefs,mathrsfs,tikz,lineno}
\usetikzlibrary{decorations,decorations.pathmorphing}

\usepackage[color]{showkeys}
\definecolor{gray}{gray}{0.5}

\newcommand\Z{\mathbb{Z}}

\newcommand\Q{\mathbb{Q}}

\newtheorem{lemma}{Lemma}[section]

\newtheorem{proposition}[lemma]{Proposition}
\newtheorem{theorem}[lemma]{Theorem}
\newtheorem{corollary}[lemma]{Corollary}
\newtheorem{example}[lemma]{Example}

\newtheorem{maintheorem}{Theorem}

\theoremstyle{definition}
\newtheorem{remark}[lemma]{Remark}
\newtheorem{definition}[lemma]{Definition}
\newtheorem{question}[lemma]{Question}

\font\manfnt=manfnt

\newcommand\twoheaddownarrow{\hbox to 0pt{\raisebox{0.3ex}{$\downarrow$}}
  \hbox to 0pt{\raisebox{-0.2ex}{$\downarrow$}}\phantom\downarrow}

\begin{document}

\title{Iterated identities and iterational depth of groups}

\author{Anna Erschler}
\address{A.E.: C.N.R.S., DMA, ENS, Paris, France}

\thanks{The work is supported by the ERC  grant 257110
  ``RaWG'' and  the ANR ``DiscGroup: facettes des groupes discrets'' }

\begin{abstract} Given word on $n$ letters, we study groups which satisfiy "iterated identity" $w$, meaning that for all $x_1, \dots, x_n$ there exists $m$ such that $m$-the iteration of $w$ of Engel type, applied to $x_1, \dots, x_n$, is equal to the identity. We define bounded groups  and  groups which are fractal with respect to identities. This notion of being fractal  can be viewed as a self-similarity conditions for the set of identities, satisfied by a group.   In contrast with torsion groups and  Engel groups, groups which are fractal with respect to identities appear among finitely generated elementary amenable groups.
We prove that  any polycyclic, as well as any metabelian group is  bounded  and we compute the iterational depth for various  wreath products.
We study the set of iterated identities, satisfied by a given group, which is not necessarily a subgroup of a free group and not necessarily invariant under conjugation, in contrast with usual identities.
Finally, we discuss another notion of iterated identities of groups, which we call solvability type iterated identities, and its relation  to elementary classes of varieties of groups.

\end{abstract}

\maketitle


\section{Introduction}

Given a word $w=w(x_1, \dots, x_n)$ on $n$ letters, we recall that a group $G$ satisfies an identity $w$ if for any $y_1, \dots, y_n \in G$ it holds
$w(y_1, \dots, y_n)$. A well-known identity on one single letter is $w(x)=x^p$. For $p$ being large enough, it is known that
there are infinite finitely generated groups satisfying this identity:
 that is,  infinite finitely generated groups such that $g^p=e$ for any $g\in G$.
A larger class of groups consists of groups that are $p$-groups, that is, $G$ are such that for any $g\in G$ there exists $k$ such that $g^{p^k}=e$.
A first group of this kind was constructed by Golod (\cite{golod}), further examples include  Grigorchuk group \cite{grigorchukfirstgroup}, Gutpa-Sidki groups \cite{guptasidki});
a large family of $p$-torsion groups of unbounded torsion appear as quotient groups of hyperbolic groups (see Gromov \cite{gromov}).
The groups above do not satisfy the identity $w=x^p$, that is, there exist $g\in G$ such that $w(g)\ne e$, but they  satisfy some iteration of
$w$: for any $g\in G$ there exists $k$ such that $w(w(\dots w(g )\dots))=e$.

What happens with other identities $w$ if we consider groups that satisfy some iteration of $w$? $w(x)=x^k$ being the only identity in one letter, so we are mostly interested in the case where $n\ge 2$. In this situation, in contrast with one letter case, there are several possible notion of "iterated identity", some of which we discuss in this paper.

Before giving the definitions, we mention possible motivation to study iterated identities of groups:

 In the case when a group satisfying an identity $w$ are well-understood, the groups satisfying an iterated identity $w$ can be much more complicated:
consider e.g. $w(x_1,x_2)=x_1x_2x_1^{-1}x_2^{-1}$; groups satisfying $w$ are Abelian, whereas the structure of Engel groups 
 is far from being understood;
in the case when groups satisfying an identity are already tricky, like Burnside groups satisfying an identity $w=x^p$, the groups satisfying the iteration of the identity form a larger, and possibly even trickier, class of groups. 

Moreover, it might be possible to construct groups satisfying additional natural properties  among groups satisfying iterated identities, such that groups satisfying the same identity do not allow: for example,  there exist {\it residually finite} infinite finitely generated torsion groups (\cite{golod, aleshin, grigorchukfirstgroup, guptasidki},) and residually finite  infinite finitely generated non-nilpotent Engel group (Golod, \cite{golod}), while any finitely generated residually finite torsion group of bounded torsion is finite (follows  the solution by Zelmanov
of so called restricted Burnside problem \cite{zelmanov});
 any finitely generated residually finite group that satisfies a fixed bounded iteration of the identity $ab=ba$ is nilpotent (Wilson, \cite{wilsonconditions}, Theorem 2, the proof
 uses classification of finite simple groups).

 Given a group $G$, it seems interesting to describe all iterated identities that $G$ satisfies.
It may happen that the group does not satisfy any non-trivial identity, but does satisfy non-trivial iterated identities
(e.g. first Grigorchuk group and  many other groups acting on a rooted tree, does not satisfy any identity by a result of Abert \cite{abert}; yet it is well known that this group is a torsion $2$ group, that is, satisfy the iteration of $w(x_1)=x_1^2$ \cite{grigorchukfirstgroup}).
The set of iterated identities on $n$ of variables does  not necessarily form a subgroup of the free group $F_n$ (see Example \ref{examplenotproduct}) and it is not necessarily
invariant with respect to conjugation (see Example \ref{conjugatenotidentity}), in contrast with usual identities, which is one of the difficulties to study this set. There are various invariants of groups which can be defined in terms of the set of all iterated identities satisfied by a group, such as boundedness and fractalness with respect to identities, which we discuss in this paper.

 In some situations iterated identities and iterated quasi-varieties provide rougher invariants, for example there are various iterated identities, satisfied by groups which do not satisfy any non-trivial identity.
In other situations iterated identities
and iterated quasi-varieties can  provide  softer invariants: it may happen that two groups $G$ and $H$ that generate distinct varieties of groups,
but iterated quasi-varities generated by these groups (both Engel type iterated quasi-varieties for a fixed (small) number of variables which hold in a given group as well as solvability  type iterated quasi-varities, which we define
in Section \ref{sectionfurthernotions}) are the same;

Let $w=w(x_1, \dots, x_n)$ be a word on $n$ letters, $n\ge 1$.

\begin{definition} \label{definitionEngel}
We say that a group $G$ satisfies $E$-type (or Engel type) iterated identity $w$ if
for any $x_1, \dots, x_n \in G$ there exists $N$ such that  $n$-th iteration of $w$ on the first variable
$$
w_{\circ N} (x_1, \dots, x_n) =w(w(\dots(w(x_1, x_2, \dots, x_n), x_2, \dots, x_n), x_2, \dots, x_n))=e.
$$

\end{definition}

Some iterated $E$ type identities  on two variables are a particular case of a closely related definition of "correct sequences" in terminology of \cite{bandmanandco, plotkinural, guralnikplotkinshalev}:
 a sequence $u = u_1,u_2,u_3,\dots ,u_n,\dots$  of elements from  a free groups on two variables $F_2$ is said to be {\it correct} if the following conditions hold:
(i) For every group $G$ and elements $a,g \in G$, we have $u_n(a,1) = e$ and $u_n(1,g) = e$ for all sufficiently large $n$.
(ii) if $a,g$ are elements of $G$ such that $u_n(a,g) = 1$, then for every $m > n$ it holds $u_m(a,g) = e$.

The words we consier in our definition \ref{definitionEngel}  are not necessarily  products of commutators, and thus do not need in general to satisfy (ii).
It is also essential that in definition \ref{definitionEngel} we do not specify first values of the sequences of iterations.

 In \cite{brandlwilson, braywilsonwilson, bandmanandco, ribnere} one describes some sequences of iteration of an identity   which characterize finite solvable groups (we recall these results   in subsection \ref{subsolvable}).
 The situation with infinite groups is essentially different. It is clear for example that no finite set of iterated identities can characterize infinite solvable groups.
 Instead of searching some iterated identity, we are interested in properties of all iterated identities which hold in a given group or a class of groups, finite or infinite, and we are mostly interested in finitely generated infinite groups.

It seems particular interesting to descibe iterated identities for some class of groups, that have some other manifestations of self-similarity, such as
 Grigorchuk groups;
 Golod groups \cite{golod};
Basilica group and other iterated monodromy groups \cite{nekrashevychbook}; branch groups \cite{grigorchuk2000}.

But it seems also interesting to understand the structure of iterated identities for groups without natural self-similarity structure, for example free and not free groups of various varieties.

The questions we want to address: which iterated identities hold for a given group or class of group? Are "quasi-varitety" defined by a set of iterated varieties of two given groups coincide?
For which groups it is essential to consider iterated identities rather than identities, that is, in which cases it is essential to allow unbounded number of iteration for a given identity:
we say that the iterated identity is bounded if this guaranteed by a bounded number of iteration. We say that the group is bounded, if it has at least one non-trivial iterated identity, all of its iterated identities are bounded and, moreover, the number of iteration for each of them is bounded by a common upper bound (see definitions \ref{defidbounded} and
\ref{defbounded} for more details).

The paper has the following structure.
Section \ref{sectionbasic} contains preliminaries about iterated identities. We give first examples of iterated identities. 
 In Example \ref{examplenotproduct} we show  that a product of
iterated identities  (in a symmetric group) is not necessarily an iterated identity.

In Section \ref{sectionbounded} we define iterational depth of an iterated identity and iterational depth of a group, as well as bounded groups and  groups which are fractal with respect to identities.

In Section \ref{sectionfractalexamples} we provide examples of finitely generated elementary amenable groups which are fractal with respect to identities (see 
Example \ref{exampleelementaryamenable}). We show that a conjugate of an iterated identity of a finitely generated elementary amenable group is not necessarily an iterated identity of this group (see Example  \ref{conjugatenotidentity}).

In Section \ref{sectionnormalsubgroups} we discuss the behavior of iterational depth with respect to quotients. We show in particular that any polycyclic group is bounded.

In Section \ref{sectionwreath} we compute the iterational depth of  wreath products of cyclic groups. We show that

-- The iterational depth of $\Z \wr \Z$ is equal to $2$ (see Proposition \ref{propositionwithz})

-- The iterational depth of     $\Z \wr \Z/R \Z$ is equal to $k+1$, $k$ being the maximum of $\alpha_i$, where $  R= \prod p_i ^{\alpha_i}$, $p_i$ are prime numbers, $\alpha_i \in \Z_+$ ( see Proposition \ref{propositionzwithfinite}).

 We show that a subgroup of a finitely generated bounded group is not always bounded.

In Section \ref{sectionmetabelian} we prove

\begin{maintheorem} [Theorem   \ref{theoremmetabelian}  ]
Any finitely generated metabelian group is bounded.
\end{maintheorem}

In Section  \ref{sectionfurthernotions}  we define another notion of iterated identities, which we call solvability type (or $S$ type) iterated identities and we discuss the relation between iterated
quasi-varieties and elementary classes generated by varieties of groups. In Example \ref{grigorchuknotelementary} we consider first Grigorchuk group and show that a group satisfying an iterated identity $w(x_1)=x_1^2$ does not necessarily belong to the elementary class of groups satisfying bounded iteration of this identity.

 In Subsection \ref{openquestions} we list some open problems.

\subsection{Aknowledgements}. The author is grateful to Rostislav Grigorchuk for his remarks on the preliminary version of this paper.

\section{basic facts and examples} \label{sectionbasic}

\begin{remark} Suppose that $w(x_1, \dots, x_n)$ is freely equivalent to $w'(x_1, \dots, x_n)$. Then a group $G$ satisfies an $E$ type iterated identity $w$ if and only if it satisfies an iterated identity $w'$.
\end{remark}

\begin{proof}
It is clear that $w_{\circ d}(x_1, x_2, \dots, x_n) = w'_{\circ d}(x_1, x_2, \dots, x_n)$
 for any $G$, any $x_1, \dots, x_n \in G$ and any  $d\ge 1$.
\end{proof}

Given a group $G$, we can therefore speak about subset $\Omega_n \subset F_n$ of  $E$ type iterated identities of $G$ on $n$ variables. We call this subset 
$\Omega_n$ {\it characteristic subset} of $G$.

The following is straightforward

\begin{lemma}
\begin{enumerate}

\item If $G$ satisfies an $E$-type iterated identity $w$, then
any subgroup of $G$ has the same property.

\item If $G$ satisfies an $E$-type iterated identity $w$, then
any quotient of $G$ has the same property.

\item  If $G_1 \subset G_2 \subset G_3 \dots$, and each $G_i$ satisfies an $E$-type identity $w$, then the union
$G= \cup_{i=1}^\infty G_i$ satisfies the same property.

\end{enumerate}
\end{lemma}

 \begin{example}
 Let $N$ be a nilpotent group,
   $G= N \wr \Z/p \Z$, $w(x_1, x_2)= (x_1x_2 x_1^{-1}x_2^{-1})^p$. Then $G$ satisfies an Engel type identity $w$.
   
 \end{example}
 
 \begin{proof}
 First observe that $w$ is an iterated identity of $N$. Indeed, $k$-th iteration of $w$ belongs to the $k$-th group in a central series of $N$.
 Observe also  that for any $x_2 \in G$ and any $x_1 \in \Z/p\Z ^\infty \subset  N \wr \Z/p \Z$, it holds $(x_1x_2 x_1^{-1}x_2^{-1}) \in (\Z/p \Z)^\infty$, and thus
 $w(x_1, x_2) =e$.
 \end{proof}

 \begin{example}  \label{examplefreesubgroup}
 i) Let  the word $w(x_1, x_2, \dots, x_n)$ be not freely equivalent to the empty word. Then for any $k\ge 1$ its iteration
 $w_{\circ k}(x_1, x_2, \dots, x_n)$ is not freely equivalent to the empty word.
 
 ii)  Let $G$ be a group that contains a non-Abelian free subgroup. Then $G$ does not satisfy any non-trivial $E$ type iterated identity: for any $w(x_1, x_2, \dots, x_n)$ which is not freely equivalent to the empty word, $G$ does not satisfy $w$.
 \end{example}
 
 \begin{proof} To prove i) and ii), 
 it is sufficient to show that the free group $F_2$ does not satisfy  the iterated identity $w$, if $w(x_1, x_2, \dots, x_n)$  is not freely equivalent to the identity word.
 First observe that $w$ is not freely equivalent to some word in $x_2, \dots, x_n$. 
 
 Indeed, if $G$ is any group satisfying $E$ type iterated identity $w(x_1, x_2, \dots, x_n)$ such that
 $w(x_1, x_2, \dots, x_n) = w'(x_2, \dots, x_n)$, then all iteration of $w$ are equal to $w'(x_2, \dots, x_n)$. Hence,  If $w$ is an $E$ type iterated identity of $G$, then $w(x_1, x_2, \dots, x_n) = w'(x_2, \dots, x_n)$ is identity of $G$. If $G$ is a free group, this implies that $w$ is freely equivalent to trivial word.
 
 Now suppose that $w$ is not freely equivalent to some word in $x_2, \dots, x_n$.  Since $F_2$ contains $F_n$ as a subgroup, it suffices to show that $w$ is not an iterated identity of $F_n$. Let $x_1, \dots, x_n$ be free generators of $F_n$. Observe that $w(x_1, \dots, x_n)$ does not belong to the subgroup generated by $x_2, \dots, x_n$. Therefore,
 $w(x_1, \dots, x_n)$, $x_2, \dots, x_n$ generate a free group of rank $n$. Applying the same argument again, we conclude that
 any iteration of $w(x_1, \dots, x_n)$ together with $x_2, \dots, x_n$ generate a free group of rank $n$. In particular, no such iteration is trivial, and hence $w$ is not an iterated identity of $G$.
 
 \end{proof}

\subsection{Known results about finite solvable groups} \label{subsolvable}

Below we use the following notation: 
$w^y= y^{-1}w y $, $w^{-y}=y^{-1} w^{-1} y$, $[u,w]=u^{-1} u^w = u^{-1} w^{-1} u w$,  $[w_1, w_2, \dots w_k] =[w_1, w_2], w_3 \dots] $ .

  One of the examples constructed by Brandl and Wilson show in \cite{brandlwilson} 
is a sequence of words in $4$ letters $X, Y, Z, T$: $s_1(X,Y, Z,T)=Y$, $s_{k+1}$ is defined recursively by $s_{k+1}(X,Y,Z,T) = [X, s_k(X,Y,Z,T), s_k(X,Y,Z,T), Z]^T$.
A sequence of words $s_n$ in variables $x,y$ is defined  by Bray, Wilson and Wilson  in \cite{braywilsonwilson} by the rules $s_0=x$ and $s_{n+1}=[s_n^{-y},s_n]$ .  Bandman,Greuel, Grunewald, Kunyavski{\u\i}, Pfister,  and Plotkin study in  (\cite{bandmanandco}) the sequence
$s_1=x^{-2}y^{-1}x$ and $s_{n+1}= [x s_n^{-1} x^{-1}, y s_n^{-1} y^{-1}]$
(in \cite{bandmanandco}  the notation used for the commutator $[u,w]   =  u w u^{-1} w^{-1}$, different from the one we use and from that of the \cite{brandlwilson, braywilsonwilson, ribnere }).
Ribnere constructs in \cite{ribnere} various  sequences of the form $s_{k+1}=[s_k^f, s_k^g]$.

In all these examples it is clear that for any solvable group of solvability length at most $k$ we have $s_k=e$. The results of above mentioned papers is  that
for finite groups the converse statement holds: any finite group such that $s_k=e$ for any choice of variables and any sufficiently large $k$ is solvable. (The idea of the proofs is that
if there would exist a non-solvable finite group $G$ with this property, then we could chose $G$ with this property among minimal simple finite groups. The authors use 
then the characterization of minimal finite groups due to Thompson (\cite{thompson1, thompson456} to find a simple group which violate this property and to get a contradiction).

\begin{remark}  \label{remarkgolod}
In all the examples above, it is clear that there exist (infinite)  finitely generated non-solvable groups $G$ such that $s_k=e$ for all sufficiently large $k$. Indeed, if $G$ is such that
any $4$ generated group of $G$ is solvable (e.g.  $G$ is one of Golod groups constructed in \cite{golod}), then it is clear that $G$ has this property.

\end{remark}

\begin{remark} \label{reformulationsolvable}
Put
$$
w_{\rm BW}(x_1, x_2, x_3, x_4)=   [x_2, x_1, x_1, x_3]^{x_4} ;
$$
$$
w_{\rm BWW}(x_1,x_2)=  [x_1^{-x_2},x_1]     = [x_2^{-1} x_1^{-1}x_2, x_1 ] = x_2^{-1} x_1 x_2 x_1^{-1} x_2^{-1} x_1^{-1} x_2 x_1;
$$
 $$
 w_{\rm BGGKPP}(x_1, x_2,x_3) = [x_2 x_1^{-1} x_2^{-1}, x_3 x_1^{-1} x_3^{-1}] = x_2 x_1 x_2^{-1}      x_3 x_1 x_3^{-1}     x_2 x_1^{-1} x_2^{-1}  x_3 x_1^{-1} x_3^{-1}    
 $$
 
The results above imply in our terminology that a finite group satisfies an $E$ type iterated identity $w_{\rm BW}$,  (or $w_{\rm BWW}$,  or $w_{\rm BGGKPP}$ etc) if
and only if it is solvable; and it is clear that any solvable group, finite or infinite, satisfy these identities.
Observe however that a class of (infinite) groups satisfying the iterated identity  $w_{\rm BGGKPP}$ might  be smaller than the class of groups, satisfying the corresponding
correct sequence from \cite{bandmanandco}, cited 
above.

\end{remark}

 \subsection{Symmetric groups}
 
 The following example shows that a product of $E$ type iterated identities is not necessarily an $E$ type iterated identity.
 
 \begin{example} \label{examplenotproduct}
 
 Consider a finite symmetric group $S_n$. Let $m$ be the product of prime numbers that are not greater than $n$. Suppose that $n \ge 6$.
 Put $w_1(x_1, x_2, x_3, x_4) = x_2 x_1^m x_2^{-1}$, $w_2(x_1, x_2, x_3, x_4) = x_3 x_1^m x_3^{-1}$, $w_3(x_1, x_2, x_3, x_4) = x_4 x_1^m x_4^{-1}$.
 Then $w_1, w_2, w_3$ are iterated identities of $S_n$, while their product $w = w_1 w_2 w_3$
 is not an iterated identity of $S_n$.
 
 \end{example}
 
 \begin{proof}
 Observe that for any $g\in S_n$ there exists $k\ge 1$ such that $g^{m^k}=e$.
 Observe that $w_{1, \circ d} (x_1, x_2, x_3, x_4) = x_2^d x_1^{m^d} x_2^{-d}$, and hence $w_1$ (and, analogously, $w_2$ and $w_3$) is an iterated identity of $S_n$.
 
 Now we prove  that $w$ is not at iterated identity of $S_n$.  The elements of the set of cardinality $n$ we number by $1$, $2$, \dots, $n$.
  Consider $x_1 \in  S_n$ which is a product of a cycle of length four: $1 \to 2 \to 3 \to 4 \to 1$ and of length two: $5 \leftrightarrow 6$. 
  Note that  $m$ is an even number which is not divided by four, and hence $g=x_1^m$ is a product of two cycles of length two: $1 \to 3 \to 1$ and $2 \to 4 \to 1$.
 Observe that $x_1$ is an even permutation. Let us show  that there exist $x_2, x_3, x_4 \in S_n$ such that
 $x_1 = x_2 g x_2^{-1}  x_3 g x_3^-1   x_4 g x_4^{-1}$. Indeed,
 $x_1$ is a product of the following cycles of length two: $5 \leftrightarrow 6$, 
 $1 \leftrightarrow 2$,
  $1 \leftrightarrow 3$ ,
   $1 \leftrightarrow 4$ (here and in the sequel we multiply the cycles from left to right). Therefore, It is also a product of the following cycles of
 length two
 $$
 x_1 =  \left( (5 \leftrightarrow 6), (1 \leftrightarrow 2) \right)    \left(   (1 \leftrightarrow 3), (5 \leftrightarrow  6)  \right)   \left(  (5 \leftrightarrow 6)  (1 \leftrightarrow 4) \right).
 $$
 Observe that a product of two cycles on any two pairs of four disjoint elements is conjugate to $g$ in $S_n$.
 We have therefore shown that there exists $x_2, x_3, x_4 \in S_n$ such that
 $$
 x_1 = x_2 x_1^m x_2^{-1} x_3  x_1^m x_3^{-1} x_2 x_1^m x_2^{-1}. 
 $$
 In other words, $x_1 = w(x_1, x_2, x_3, x_4)$, and hence $w_{\circ d} (x_1, x_2, x_3, x_4) = x_1 \ne e$ for any $d \ge 1$.
 This shows that $w$ is not an iterated identity of $S_n$.

 \end{proof}

\section{Bounded groups, fractal groups, iterational depth of a group} \label{sectionbounded}

\begin{definition} \label{defidbounded}
Given an Engel type iterated identity $w$ of a group $G$ let us say that $G$ has {\it  iterational depth} at most $D$ with respect to $w$ if for any $x_1, \dots, x_n \in G $ there exists
$d\le D$ such that  $d$-th Engel type iteration $w_{\circ d}(x_1, \dots, x_n)=e$ in $G$. The infimum of such $D$ we call {\it iterational depth} of $w$ in $G$ and denote
by $s(w,G)$ (or $s(w)$ for short).
 If iterational depth of $w$ in $G$ is finite, we call such iterated identity {\it bounded}.

We also say that $G$ has {\it iterational depth} at most $d$ if the  depth  of $G$ with respect to any Engel type iterated identity that it satisfies,  is at most $d$.
The supremum of such $d$  over all  non-trivial (that is, not equivalent to the empty word in a free group) Engel type iterated identities $w$ we denote by $s(G)$ and call it {\it iterational depth} of $G$. If $G$ does not satisfy any non-trivial iterated identity, we put $s(G)=-\infty$.

\end{definition}

\begin{definition} \label{defbounded}
 We say that $G$ is bounded if $s(G)$ is finite (that is,  $1\le s(G)<\infty$). If $s(G)=+\infty$ we say that $G$ is fractal with respect to identities.
\end{definition}

\begin{remark} \label{boundedandidentities}
i)  Any bounded group satisfies a non-trivial identity.

ii) If $G$ is such that $s(G)<\infty$ then $G$ is bounded if and only if $G$ satisfies a non-trivial identity.
\end{remark}

\begin{proof}

i) By definition, since $G$ is bounded, there exists a non-trivial iterated identity $w$ satisfied by $G$. Since $G$ is bounded, it implies that there exists $D$ such that
for any $x_1,  x_2, \dots, x_n$  it holds  $w'(x_1, \dots, x_n)=e$ for at least one $w'$ among $w$, $w_{\circ 2}$, \dots, $w_{\circ D}$.
If $w$ is not freely equivalent to the empty word, then none among $w$, $w_{\circ 2}$, \dots, $w_{\circ D}$ is freely equivalent to a free word (see i) of Example \ref{examplefreesubgroup}).
If any tuple of elements of a group satisfy at least one among a finite list of identities, then the group satisifies a non-trivial identity.

ii) If $G$ satisfies a non-trivial identity $w$, then $G$ satisfies the (non-trivial) iterated identity $w$. Therefore, by definition of iterational depth $s(G)\ne - \infty$.
Therefore, if $s(G)<\infty$, then $s(G)$ is a finite number, and thus $G$ is bounded.

If $G$ is bounded, then by the first part of the remark we know that $G$ satisfies a non-trivial identity. Therefore, $s(G) \ne \infty$, and thus $G$ is bounded so far
as $s(G) \ne - \infty$.

\end{proof}

\begin{example}  \label{examplefinite}

Any finite group is bounded. 

\end{example}

\begin{proof}
Since $G$ satisfies the $E$ type iterated identity $w$, we know that for any $x_1, \dots, x_l$ in $G$ there exists
$n$ such that $w_n(x_1, \dots, x_l)=e$.  Put $s = \max n$, where maximum is taken over all possible choices of $x_1, \dots, x_l$ in $G$.

\end{proof}

\begin{example} \label{exampleabelian}
$s(\Z^d)=1$. Similarly, if $G$ is an Abelian group which admits an element of infinite order, then $s(G)=1$. 
\end{example}

\begin{proof}
Observe that if $w(x_1, \dots, x_n)$ is a $E$ type iterated identity of $G$, then for any $i$ the total number of $x_i$ in $w$ is equal to zero.
(Indeed, otherwise we  put $x_j=e$ for any $j\ne i$ and let $x_i$ be an element of infinite order in $G$, we observe that
$w_{\circ d}(x_1, x_2, \dots x_n) \ne e$ for any $d\ge 1$ and
 we get a contradiction).
Since $G$ is Abelian, we observe that for any $x_1, \dots, x_n \in G$ it holds $w(x_1, \dots, x_n)=0$.

\end{proof}

\begin{remark}
If $G$ is a finitely generated Abelian group, then $s(G)$ is finite (indeed, in this case it is finite in case if $G$ is finite by Example \ref{examplefinite} and is equal to one
in case if $G$ is infinite by Example \ref{exampleabelian}).

It is clear that  an infinitely generated Abelian group can have an infinite identity depth, as for example an infinite direct some $\sum_{i\ge 1} \Z/ p^i \Z $ which satisfies an iterated identity $w(x_1)=x_1^p$ and does not satisfy any identity of the form $w'(x_1)=x_1^{p^k}$.
\end{remark}

\section{Fractal examples among elementary amenable groups} \label{sectionfractalexamples}

\begin{example} \label{exampleelementaryamenable}
i)  Let $G$ be an extension of  group of upper uni-triangular infinite square matricies   with integer coefficients by a cyclic group (that acts on $\Z \times \Z$ by
$z_1, z_2 \to z_1+1, z_2+1$). Consider $\bar{w}(x_1, x_2, x_3) = [x_1, [x_2, x_3]]]$.
Then $G$ satisfies the iterated identity $\bar{w}$. For any iterated identity $w$ of $G$ it holds
$s(w, G) =\infty$. In particular, $G$ is fractal with respect to identities.

ii) Moreover, there exists a continuum of non-isomorphic finitely generated elementary amenable  groups that are fractal with respect to identities, and which do not satisfy any identity.

\end{example}

\begin{proof}
i) Denote by $M_\infty$  the group of infinite upper uni-triangular matrices  with integer coefficients (which we denote by $m_{i,j}$, $i, j\in Z$). This group is generated by  upper triangular matrices $m_i$, $i\in \Z$  such that $m_i (i, i+1)=1$
and $m_i(j,k)=0$ if $j\ne k$ and $(j,k)\ne (i,i+1)$.
We denote by $\phi$ the  group homomorphism that sends $m_i$ to $m_{i+1}$. The group $G$  is by definition the extension of $M_\infty$ by $\phi$.
It is clear that $G$ is a finitely generated group, indeed, it is generated by $\phi$ and $m_0$. 
$G$ is elementary amenable, since it is an extension of a locally nilpotent group $M_\infty$ by a cyclic group.

\begin{lemma} \label{lemmaiteridw}
$G$ satisfies the iterated identity $\bar{w}=\bar{w}(x_1, x_2, x_3) = [x_1, [x_2, x_3]]]$.
\end{lemma}

\begin{proof}
First observe that for any $x_2, x_3 \in G$ it holds $[x_2, x_3] \in [G,G] \subset M_\infty$  and for any $x_1 \in G$ it holds
$\bar{w}(x_1, x_2, x_3) \in [G,G] \subset M_\infty$.
Take some $x_1, x_2, x_3$ and consider $x=[x_2, x_3]$ and $y=\bar{w}(x_1, x_2, x_3)$. Since $x, y \in M_\infty$, they generate a subgroup which is nilpotent, and hence Engel.
This shows that there exists $d\ge 1$ such that $d$-th commutator of the form $ [\dots  [[x,y], y] \dots , y]$ is equal to $e$. This implies that
$\bar{w}_{\circ (d+1)} (x_1, x_2, x_3)=e$.

\end{proof}

Now we want to show that for any iterated identity $w$ of $G$ it holds $s(w, G) = \infty$.

\begin{lemma} \label{lemmaMwithoutidentities}
For any (non-iterated)  non-trivial identity $w$ there exists $n$ such that  the group 
$M^{(n)}$ of $n \times n$ upper uni-triangular matrices with integer coefficients does not satisfy $w$.
\end{lemma}

\begin{proof}
Any torsion-free finitely generated nilpotent group can be imbedded as as subgroup in  $M^{(n)}$ for some $ n\ge 1$ (by a result of Jennings, see
\cite{hallnilpotentgroups}).
Let $w$ be a word which is not freely equivalent to the empty word.
It is sufficient to check that there exists a torsion-free finitely generated   nilpotent group which does not satisfy $w$.
The latter follows from the fact that the free group $F_2$ does not satisfy $w$, and that free groups are residually torsion-free nilpotent (\cite{magnus1935}).

\end{proof}

Lemma \ref{lemmaMwithoutidentities} and Remark \ref{boundedandidentities} imply that for any iterated identity $w$ of $G$ it holds
 $s(w,G)=\infty$.    Since $\bar{w}$ is a non-trivial iterated identity of $G$ (by Lemma \ref{lemmaiteridw}), this implies, in particular, that $G$ is fractal with respect to identities.

ii) Let $G$ be the group constructed in the proof of i). Let $G_1$ be any finitely generated solvable group of step $3$ such that the quotient over the the center
$G/C(G)= \Z \wr \Z$.
Consider the direct product $G'=G_1$. Observe that $G_1$, and hence $G'$ satisfies the iterated identity $w$ from the proof of i).
Since $G$ is a quotient of $G'$, this implies that $s(w, G') \ge s(w, G) \ge \infty$, and hence $s(G')= \infty$, in other words $G'$ is fractal with respect to identities.
Observe also that $C(G)$ is trivial, and hence $C(G+G_1) = C(G)$.
By \cite{hall1954}, Section 3, we know that any countable Abelian group is a center of some group $G_1$ as above. Therefore, any countable Abelian $A$ group there exists
$G'$, as above, with the center equal to $A$. In particular, there exists a continuum of such groups $G_1$,

Finally observe, that any $G_1$ as above does not satisfy a non-trivial identity, since $G_1$ contains $G$ as a subgroup and since $G$ does not satisfy any non-trivial identity.

\end{proof}

The following example shows that , in contrast with usual identities, a conjugate of an iterated identity is not necessarily an iterated identity.

\begin{example} \label{conjugatenotidentity}

 Let $G$ and $\bar{w}$ be as in Example \ref{exampleelementaryamenable}.
Consider $w'(x_1, x_2, x_3, x_4)= x_4 \bar{w}(x_1, x_2, x_3) x_4^{-1}$. Then $G$ does not satisfy $w'$. 

\end{example}

\begin{proof}

Let $e_{i,j} (\lambda)$, $j>i$ denotes the elementary upper-diagonal unitary matrix $m$, such that $m_{i,j}= \lambda$ and all the other entries of $m$ over diagonal
are zero. Denote by $e_{i,j} = e_{i,j} (1)$.

\begin{lemma} Let $\alpha$ be an automorphism the group of infinite upper uni-triangular matrices $M_\infty$ such that  $\alpha(M)_{i,j} = M_{i+1, j+1}$ 
Given $y \in M_\infty$ consider the map $m_y: M_{\infty} \to M_{\infty}$ which sends $x$ to $ \alpha(x y x^{-1} y^{-1})$.

If  $y=e_{-1,0} e_{0,1}(-1) e_{-1, 1}(-1)$ and
$x=e_{-1,0}$, then $m_{y, \circ d} (x)\ne e$ for any $d \ge 1$. 
Moreover,  $m_{y, \circ d} (x)\ne e$  is an elementary matrix with exactly one non-zero value off the main diagonal, at $(-d, 0)$. 

\end{lemma}

\begin{proof}
First note that $\alpha$ preserves the lower central series   $M=M_0 \supset M_1 \subset M_2 \dots $ of    $M_\infty$
(here $M_{i+1}= [M, M_i]$), and therefore $m_{y, \circ d} (x) \in M_d$ for any $d\ge 1$.
This implies that the entries on $d$ diagonal above the main diagonal (that is, the entries at position$(k, k+l)$, $l \le d$) of the matrix
$m_{y, \circ d} (x)$ are zero.

Now prove by induction on $d$ the fact that $m_{y, \circ d} (x)\ne e = e_{-(d+1), 0} (\lambda)$, for some $\lambda \ne 0$.

The base $d=0$ follows from the definition of $x$.

Suppose that  $m_{y, \circ d} (x)\ = e_{-(d+1), 0} (\lambda)$, for some $\lambda \ne 0$.

Observe that all non-zero entries of $[e_{-(d+1) , 0} (\lambda), y]$ off the main diagonal belong to the square  $i \ge -(d+1), i \le 1$, $j \ge -(d+1), j\le 1$.
Indeed, observe that if a matrix belongs to this square, then its inverse also belongs to this square, and if two matrices belong to this square, then
their product belongs to this square.
Observe also that $[e_{-d+1, 0} (\lambda), y]$ belongs to $M_{d+1}$, and this is an uni-triangular matrix with  $d+1$ zero diagonals above the main diagonal.
We conclude that the only possible value of $[e_{-d+1, 0} (\lambda), y]$ off the main diagonal at $i=-(d+1)$, $j=1$. Let us show that this entry is non-zero.
To do this observe  that $e_{-d+1, 0}$ does not commute with $y$. Indeed, the entry at $-(d+1,1)$ of $e_{-(d+1) , 0} (\lambda) y$ is equal to $-\lambda$, and that the
the entry at $-(d+1,1)$ of $y e_{-(d+1) , 0} (\lambda) $ is equal to zero.

We have shown that $[e_{-d+1, 0} (\lambda), y]$ is $e_{-(d+1, 1)}(\lambda')$, for some $\lambda' \ne 0$. This implies that
$\alpha([x,y]) = \alpha (  [e_{-(d+1), 0} (\lambda), y]     ) = e_{-(d+2), 0} (\lambda')$, and this concludes the proof of the induction step and the proof of the lemma.

\end{proof}

Now take $x_2, x_3 \in G$ such that $[x_2, x_3] =  y$. Such $x_2$ and $x_3$ do exist: indeed, take $x_2= m_0$, $x_3= \phi$, observe that
$x_2 x_3 x_2^{-1} x_3^{-1}= m_0  \phi m_0 \phi^{-1} = m_0 m_1$. 
Put  $x_1=x=  e_{-1,0}$ and $x_4 =\phi$. 

Observe that for any $ x\in M$ it holds $w'(x, x_2, x_3, x_4) = m_y(x)$. Therefore
$w'_{\circ }(x_1, x_2, x_3, x_4) =m_{y, \circ d} (x_1) \ne e$ for any $d\ge 1$.
We have therefore shown that $w'$ is not an iterated identity of $G$.

\end{proof}

\section{Identities and iterated identities of group extensions, restriction of verbal map on an invariant subgroup and dynamics of its iterations} \label{sectionnormalsubgroups}

Let $w(x_1,\dots, x_n)$ be a word and $G$ be a group. 
Fix $x_2, \dots, x_n$ and consider
a verbal map from $G$ to $G$, $x \to w(x, x_2, \dots, x_n)$. We denote this verbal map
by $\phi_w=\phi_{w,x_2, \dots, x_n}$.

Now let $N$ be a normal subgroup of $G$ and suppose that $w$ is an identity of $G/N$. Then for any $x_2, \dots, x_n \in G$ and any $x \in G$ it holds
$w(x,x_2, \dots, x_n) \in N$. In other words $\phi_{w,x_2, \dots, x_n} (x) \in N$ for any $x\in G$, and, in particular, for any $x\in N$.
We can therefore study the restriction of the verbal map $\phi_{w,x_2, \dots, x_n}$ to $N$.

\begin{remark} \label{remarkuv}
Observe than any word $w=w(x_1,\dots, x_n)$ can be written as
$w(x_1,\dots, x_n)  =u(x_1, \dots, x_n)     v(x_2, \dots, x_n)$,
where  the word $u$ has the following form
$$
u(x_1, \dots, x_n)  =\prod_{i=1}^N   \alpha_i x_1^{l_i} \alpha_i^{-1},
$$
where $l_i \in \Z$  and $\alpha_i$ are some word in $x_2, \dots, x_n$.

If $w$ is an identity of $G/N$, then so is $v(x_2, \dots, x_n)$. 
In other words 
$v(x_2, \dots, x_n) \in N$ for any $x_2, \dots, x_n \in G$ 
\end{remark}

\begin{remark} \label{kernelinN}
Let $w(x_1, \dots, x_n)$ be a word, $G$ be a group and $N$ is a normal subgroup of $G$. Fix some $x_2, \dots, x_n$ and consider
the verbal map $\phi=\phi_{w, x_2, \dots, x_n}$ (that is, $\phi(x)= w(x, x_2, \dots, x_n)$.
Then either $\phi(n) \in N$ for any $n \in N$, or $\phi(n) \notin N$ for any $n \in N$.

\end{remark}

\begin{proof}
Use Remark \ref{remarkuv} and write $w= u(x_1, \dots, x_n) v(x_2, \dots, x_n)$. Since $u$ is a product of conjugates of $x_1$ and since $N$ is a normal subgroup, we conclude
that $\phi_u (x) \in N$ for any $ x \in N$.
If $v_{x_2, \dots, x_n} \in N$, then $\phi (x) \in N$ for any $x \in N$.

Otherwise, if $v_{x_2, \dots, x_n} \notin N$, then $\phi (x) \notin N$ for any $x \in N$.

\end{proof}

\begin{lemma} \label{lemmahomom}
Consider $w$, $v$ and $u$ defined in Remark \ref{remarkuv}. Fix $x_2, \dots, x_n$.
Consider the verbal map $\phi_u$ from $G$ to $G$ and its restriction to $N$.

If $N$ is  an Abelian normal subgroup, then the restriction of $\phi_u$ to $N$ defines a group homomorphism on $N$.

\end{lemma}

\begin{proof} Let $y, z\in N$. Observe that $(yz)^{l_i}=y^{l_i}z^{l_i}$, and hence
$$
\phi_u(yz) =\prod_{i=1}^N   \alpha_i (yz)^{l_i} \alpha_i^{-1} =\prod_{i=1}^N  \left(  \alpha_i y^{l_i} \alpha_i^{-1}    \alpha_i z^{l_i} \alpha_i^{-1} \right)=
\prod_{i=1}^N   \alpha_i y^{l_i} \alpha_i^{-1}  \prod_{i=1}^N   \alpha_i z^{l_i} \alpha_i^{-1} =\phi_u(y) \phi_u(z).
$$

\end{proof}

\begin{lemma}  \label{commutewithaction}.
We consider again $w, u, v$ from Remark \ref{remarkuv}.
Let $G/N$ and $N$ be Abelian groups. Then

i)  for any $x_2, \dots, x_n \in G$ the restriction of $\phi_u$ to $N$ commutes with the action (by conjugation) of
$G$ on $N$.

ii) for any $x_2, \dots, x_n \in G$ any any $x\in N$ $\phi_u(x)$ belongs to the linear span of $a(x)$, $a\in G/N$.
\end{lemma}

\begin{proof}

First observe that the action of $g$ on $N$ depends only on the image of $g$ in $G/N$. Indeed, if $g=h n$, $n\in N$, then
$gmg^{-1} =  h n m  n^{-1} h^{-1} = h m h^{-1}$ ($n$ commute with $m$ since $m, n \in N$ and $N$ is Abelian).
To prove i) observe that by the definition of $\phi_i$ in Remark \ref{remarkuv}
$ \phi_u(n)= \prod \alpha_i n^{l_i} \alpha_i^{-1}$, and hence  
$$
g \phi_u(n) g^{-1} = g \prod \alpha_i n^{l_i} \alpha_i^{-1} g^{-1}= \prod  g \alpha_i n^{l_i} \alpha_i^{-1} g^{-1} = \prod  g \alpha_i  n^{l_i} (g\alpha_i)^{-1}
$$ 

Since $G/N$ is Abelian, $g \alpha_i = \alpha_i g$ in $G/N$, and hence the action by conjugation on $N$ of these two elements are the same. Therefore
$$
g \phi_u(n) g^{-1} = \prod   \alpha_i g n^{l_i} (\alpha_i g)^{-1} =  \prod \alpha _i     (g n g^{-1})^{l_i}    \alpha_i^{-1} = \phi_u (g n g^{-1}).
$$

ii) follows from the definition of $\phi_u$.
\end{proof}

Two following lemmas control the depth of a group which contains a cyclic normal subgroup, finite or infinite.

\begin{lemma} \label{lemmafinitenormalsubgroup}

Let $G$ be a group and $N$ be a finite normal subgroup of $G$. The cardinality of $N$ we denote by $\# N$.
  Let $w$ be an iterated identity of $G$.
Then
$$
s(w, G) \le (s(w, G/N)+1)( \# N +1)
$$

In particular, If $G/N$ is bounded, then $G$ is bounded.

\end{lemma}

\begin{proof} Ift $w(x_1, \dots, x_n)$ be an iterated identity of $G$, then $w$ is an iterated identity of $G/N$. Denote by $D$ the iterational depth $s(w, G/H)$.
For any $x_1 \dots, x_n \in G/N$ there exists $d \le D$ such that $w_{\circ d}(x_1, \dots, x_n) =e$ in $G/N$. In other words, for any $x_1, \dots, x_n$ in $G$ here exists $d \le D$ such that $w_{\circ d}(x_1, \dots, x_n)  \subset N$.
Fix $x_2, \dots, x_n$ and consider some $x\in N$. Since $w$ is an iterated identity of $G$ there exists $m$ (depending on $x$ and $x_2, \dots, x_n$) such that
$w_{\circ m} (x, x_2, \dots, x_n) =e$. 

Let $\phi:G \to G$ be the verbal map $\phi(x)= w(x, x_2, \dots , x_n)$. 
Let $R$ denotes the cardinality of $N$.
Observe that for any $x\in G$ there exist $1\le m_1< m_2< m_3< m_{R+1} $ such that
$\phi_{\circ m_j} (x) \in N$ for any $j: 1 \le j \le R+1$ and such that  $m_1 \le  D+1$, such that $m_{j+1}- m_j \le D+1$ for any $j \le R$ and such that

Observe that  $1\le m_1< m_2< m_3< m_{R+1} \le (D+1)(R+1)$.
Therefore, there exist at least among these $R+1$ elements there exist at least two equal ones, that is, there exist $r,s: 1\le r<s \le (R+1) (D+1)=$ such that
$\phi_{\circ r} (x)=  \phi_{\circ s} (x)$, This implies that for any $t \ge 0$ $\phi_{\circ r+t} (x)=  \phi_{\circ s+t} (x)$.
This implies that if $\phi_{\circ m}(x) =e$ for some $m \ge 1$, then there exists $ m \le (R+1)(D+1)$ such that $\phi_{\circ m}(x) =e$

We have shown therefore  that  $s(w, G) \le (s(w, G/N)+1)( \# N +1)$.

Now assume that $G/N$ is bounded. Then $s(G/N)<\infty$, and thus $s_{G} < \infty$. Since $G/N$ is bounded, we know from Remark \ref{boundedandidentities} that
$G/N$ satisfies at least one non-trivial identity, which we denote by $w$. Then $G$ satisfies an identity $w^R$, and it is clear that $w^R$ is not freely equivalent to a trivial word.
Thus, we know that $s(G) < \infty$ and that $G$ satisfies a non-trivial identity. Using  \ref{boundedandidentities} we conclude that $G$ is bounded.

\end{proof}

\begin{lemma} \label{lemmacyclicquotient}
Let $G$ be a group and $N$ be an infinite normal cyclic subgroup of $G$.
For any iterated identity $w$ of $G$ it holds $s(w,G) \le 2 s(w, G/N)+1$.

 In particular, if If $G/N$ is bounded, then $G$ is bounded.

\end{lemma}

\begin{proof}

Let $w(x_1, \dots, x_n)$ be an iterated identity of $G$. Since $w$ is an iterated identity of $G/N$ and since $G/N$ is bounded, there exists $D$ such that for any $x_1, \dots, x_n$ in $G$ here exists $d \le D$ such that $w_{\circ d}(x_1, \dots, x_n)  \subset N$.

Fix $x_2, \dots, x_n$.
For any $d \ge 0$ consider the verbal map $\phi_d$,
    defined by $w_{\circ d}$.  By remark \ref{kernelinN} we know  that either $\phi_d(x) \in N$ for any $x\in N$, or $\psi_d(x) \notin N$ for any $x \in N$.
 This implies that if  $d''<d'$ and $\phi_d(e), \phi_d'' (e) \in N$, then $\phi_{d'-d''}(e) \in N$. Let $D$ be a common divisor of such $d'$. We have
 $\phi_l (e) \in N$ if and only if $l$ is divided by $D$, and , moreover, for any $x \in N$  we know that  $\phi_l (e) \in N$ if and only if $l$ is divided by $D$.
 In particular, if  $\phi_l (x) =e$ for some $x\in N$,  then $l$ is divided by $D$. Observe that $D \le d+1$, since there exists $m: 1\le m \le d+1$ such that
 $\phi_m(e) \in N$.
 
Apply Remark \ref{remarkuv} to $w_{\circ D}$ and consider $u$ and $v$ defined in this remark. For any $x_2, \dot, x_n \in G$, the restriction of the verbal map of $w_{\circ D}$ has the form
$\psi(x) = \phi_u (x) +v$, where $\phi$ is a homomorphism of (an infinite cyclic group) $N$, and $v \in N$.

We have already shown that for any $x\in N$ there exists an integer $l$ such that $\phi_{Dl}(x)= \psi_{\circ l} (x)=e$.

Since $\phi_u$ is a homomorphism from $\Z$ to $\Z$, there exists an integer $T$ such that $\phi_u(x) = Tx$. Since $\psi_{\circ m(e)}=0$ for some $m$, it follows that $v+\phi_u(v)+ \dots \phi_{u, \circ (m-1)}(v) =0$. This implies that $\phi_u(e)$ takes finitely many values, we denote this finite set by $\Omega$. Observe that for any $l\ge 1$ there exists
$\omega \in \Omega$ such that $\psi_{\circ l} (x) = T^l x+ \omega$, for any $x \in N$. Assume that $T \ne 0$. Note that then for any sufficiently large $x$ all values of
$\psi_{\circ l }(x)$ are non-zero. This contradiction shows that $T=0$.  

We know , therefore, that  $\phi_(u)(x)=0$ for any $x\in N$, and hence $\psi_{\circ l}(x)= v$ for any $l \ge 1$, and hence $v=e$. This shows that $\psi(x)=0$ for any $x \in N$, and this implies that for any $x_1, \dots, x_n$ there exists $l \le d+D \le 2d+1$ such that 
$w_{\circ l}(x_1, \dots, x_n)=e$. We have shown therefore that $s(w,G) \le 2 s(w, G/N)+1$.

Finally, observe that  $G$ satisfies a non-trivial identiy if $G/N$ satisfies a non-trivial identity. Indeed, let $w(x_1, \dots, x_n)$ be an identity of
$G/N$. Consider $w'(x_1, \dots, x_{2n})= [w*(x_1, \dots, x_n), w(x_{n+1}, \dots, x_{2n})]$. It is clear that $w'$ is not freely equivalent to an empty word
if $w$ has this property, and it is clear that $G$ satisfies the identity $w'$.
\end{proof}

Lemma \ref{lemmafinitenormalsubgroup} and Lemma \ref{lemmacyclicquotient}  imply the following

\begin{corollary}
Any finitely generated polycyclic group is bounded. 
\end{corollary}

\subsection{Nilpotent groups}

\begin{remark} \label{remarkuvnilpotent}
Any word $w(x_1, \dots, x_n)$ can be written as
$$
w(x_1, x_2, \dots , x_n) = \prod_{j=1}^N [ x_1^{l^j},  u_{j}] x_1^r v(x_2, \dots, x_n),
$$
where $u_j$ are words in $x_2, \dots, x_n$, $l_j$ are integers  and $r$ is the total number of $x_1$ in $w$ .
We denote 
$$
u(x_1, \dots, x_n) =  \prod_{j=1}^N [ x_1^{l^j},  u_{j}].
$$
\end{remark}

\begin{proof} Write a word $w$ as a product of words $(x_2, x_3, \dots, x_n)$ and powers of $x_1$ and prove the statement of the remark by induction on the number of terms in such product.

\end{proof}

It is not difficult to describe all $E$ type iterated identities of a nilpotent group:

\begin{example}  \label{examplenilpotent}

Let $G$ be a nilpotent group and $w(x_1, x_2, \dots, x_n)$ be a word.  Let $u, v$ be the words defined in Remark \ref{remarkuvnilpotent}
and $r$ be the total number of $x_1$ in $w$.

i) Suppose that  there exists an integer $m>0$ 
 such that for  any element $g\in G$ there exists  an integer $t >0$ such that 
such that  
$g^{m^t}=e$.
Take $m$ to be the minimal positive integer with this property.
The word $w$ is an $E$ type iterated identity of $G$ if and only if  $r$  is divided by $m$ and $v(x_2, \dots, x_n)$ is an identity of $G$.

ii) Otherwise (for example, if $G$ contains at least one element of infinite order) the word $w$ is an $E$ type iterated identity of $G$ if and only if $r=0$ 
and $v(x_2, \dots, x_n)$ is an identity of $G$.

\end{example}

\begin{proof}
 First suppose that $w$ is an $E$ type  iterated identity of $G$. Consider $x_1=x$, $x_2= x_3= \dots x_n=e$.
 Observe that for all $d\ge 1$ it holds $w_{\circ d}(x_1, x_2, \dots, x_n) = x^{r^d}$. This implies $r=0$ that under assumption of ii) and that $r$ is divided by $m$ under assumption of i).
 
Let us show that $v$ is an identity of $G$. Suppose not. Then there exists $x_2, \dots, x_n$ such that $v(x_2, \dots, x_n) \ne e$.

If $G$ satisfies the assumption of ii), take $M\ge 0$ such that $v(x_2, \dots, x_n)$ belongs to the $M$-th term of the lower central series of $G$ and 
does not belong to the $M+1$-th term of this series. Observe that for any $x_1$ and any $d \le M$ the iteration $w_{\circ d}(x_1, x_2, \dots, x_n)$ belongs to the $d$-th term of the lower central series of $G$. Observe also that  $w_{\circ (M+1)}(x_1, x_2, \dots, x_n)$ (as well as $w_{\circ d}(x_1, x_2, \dots, x_n)$ for any $d\ge M+1$) belongs to the $M$-th term of the 
lower central series but does not belongs to the $M+1$-th term. This shows that $w_{\circ d}(x_1, x_2, \dots, x_n)$ is not equal to $e$, for any $x_1$ and any $d \ge M$.
This is a contradiction with the fact that $w$ is an iterated identity of $G$.

If $G$ satisfies the assumption of i), observe that any finitely generated subgroup of $G$ is finite. 
 Consider  the following central series of $G$: $\lambda_0(G)=G$, $\lambda_{i+1}(G)$ be the normal subgroup generated  by $[\lambda_i(G),G]$ and $ \lambda_i(G)^m$. 
We consider, as before, $x_2, \dots, x_n$ such that $v(x_2, \dots, x_n) \ne e$. Put $x_1=e$. Let $G'$ be the subgroup of $G$ generated by $x_1, x_2, \dots, x_n$.
Observe that the series $\lambda_d(G')$ (of the finite group $G'$) has a finite number of non-trivial terms. Since  $v(x_2, \dots, x_n) \ne e$, there exists $M\ge 0$ such that
 $v(x_2, \dots, x_n) \ne e$ belongs to $\lambda_M(G')$ and does not belong to $\lambda_{(M+1)}(G')$.
 
 Observe that for any $x_1$ and any $d \le M$ $w_{\circ d}(x_1, x_2, \dots, x_n)$ belongs to $\lambda_d(G')$. Observe also that  $w_{\circ M+1}(x_1, x_2, \dots, x_n)$ (as well as $w_{\circ d}(x_1, x_2, \dots, x_n)$ for any $d\ge M+1$) belongs to  $\lambda_M(G')$-th but does not belongs to  $\lambda_{M+1}(G')$-th term. This
 contradicts the fact  that $w$ is an iterated identity of $G$.

 Now suppose that $r_i$ satisfy the assumption of i) and ii) correspondingly. Let us show that $w$ is an $E$ type iterated identity of $G$.

 Under assumption of ii) it is clear that $w_{\circ d}(x_1, x_2, \dots , x_n) = x^{N}$ belongs to the $d$-the term of the lower central series of $G$.
 
 If $G$ satisfies the assumption of i) 
 and $w$ satisfies the condition of i), it is clear that  $w_{\circ d}(x_1, \dots, x_n)$ belongs to $\lambda_d(G')$, where $G'$ is a  subgroup generated by $x_1, \dots, x_n$.
 Observe that $G'$ is a finite nilpotent group such that for any $g\in G'$ there exists $t$ such that $g^{m^t}=e$. This implies that  $\lambda_d(G')$ has finite length (that is, there exists a finite number of non-zero terms in this series), and therefore $w$ is an iterated identity of $G$.

\end{proof}

\begin{corollary} Let $G$ be a nilpotent group. Let $\Omega_n \subset F_n$ be the set of $E$ type iterated identities of $G$.
Then $\Omega_n $ is a normal subgroup of the free group $F_n$, for any $n\ge 1$.

\end{corollary}

\begin{proof}
If $G$ satisfies the assumption of ii) and $w_1$ and $w_2$ satisfy the conditions of ii), observe that their product $w= w_ 1 w_2$ , as well as for any conjugate
$w= w' w_1 (w')^{-1}$
it holds $w_{\circ d}(x_1, \dots, x_n)=u_{\circ d}(x_1, \dots, x_n)$ belongs to the $d$-th term of the lower central series of $G$.

Observe also that if  $G$ satisfies the assumption of i) and $w_1$ and $w_2$ satisfy the conditions of i), then  their product $w= w_ 1 w_2$ , as well as any conjugate
$w= w' w_1 (w')^{-1}$
it holds $w_{\circ d}(x_1, \dots, x_n)$ belongs to $\lambda_d(G')$, where $G'$ is a (finite nilpotent) subgroup generated by $x_1, \dots, x_n$.

\end{proof}

\section{Examples of iterational depth. Wreath products}  \label{sectionwreath}

\begin{proposition} \label{propositionwithz}
$s(\Z \wr \Z) =2$
\end{proposition}

\begin{proof}
Let $G= \Z \wr \Z$, $N= \Z^\infty \subset G$. Let $A=G/N$. We know that $A=\Z$.
Let $w(x_1, \dots, x_n)$ be an iterated $E$-type identity of $G$. Since $G$ has elements of infinite order, the total number of $x_i$ in $w$ is equal to zero, for each $i$.
Therefore, $w$ is an identity of $A$.
We consider $u(x_1, x_2, \dots, x_n)$ and $v(x_2, \dots, x_n)$ defined in Remark \ref{remarkuv}. It holds
$w(x_1, \dots, x_n)=u(x_1, x_2 \dots x_n)  v(x_2, x_3, \dots x_n)$.

We want to show that for any $x \in N$ it holds $\phi_u(x)=0$. Let $\epsilon_i$, $i\in \Z$ be the standard basis of $\Z^\Z$.
If  $\phi_u(x) \ne 0$ for some $x \in N$, then there exists $i\in \Z$ such that $\phi(\epsilon_i)\ne 0$.
Let
$$
\phi_u(\epsilon_i) = \sum_{i=-M}^M \beta_{j} \epsilon_{i+j} 
$$

By Lemma \ref{commutewithaction} we know that $\phi_u $ commutes with the action of $\Z$ on $\Z^\Z$ by conjugation. Therefore, for any $k\in \Z$ it holds

$$
\phi_u(\epsilon_k) = \sum_{i=-M}^M \beta_{j} \epsilon_{k+j} 
$$

Let $m$ be the maximal among $i$ such that  $\beta_i \ne 0$ and suppose that $m\ge 0$. 
Let 
$$
v=v(x_2,x_3, \dots, x_n) = \sum_{i=-L}^L \gamma_i \epsilon _i
$$

Take any $k >L$. Observe that the $l$-the iteration of $\phi_u v$, applied to $\epsilon_k$, has a coordinate
$\beta_m^l \ne 0$ at the basis vector $\epsilon_{k+ml}$. Therefore, any such iteration is not equal to zero, and hence
no $E$ type iteration of $w$ at $(x_1, \dots, x_n)$ is equal to $e$ in $G$. We have got a contradiction with the fact that $w$ is an iterated identity of $G$. 

Therefore,  for any $x_2, \dots, x_n$ and any $x\in N$ it holds $\phi_u(x)=0$. Take any $x_1 \in G$. Since $w(x_1, \dots, x_n \in N$, we conclude that for any $l \ge 2$ the 
$l$-the iteration of $w$ at $x_1, \dots, x_n$ is equal to $v(x_2, \dots, x_n)$. Since for some $l$ this iteration is trivial, we see that 
$v(x_2, \dots, x_n)=e$ and hence the second iteration of $w$ is trivial.

We have shown that  $s(\Z\wr Z) \le 2$. 

\begin{remark}
Take $w=(x_1, x_2)=[x_1, [x_1,x_2]]$. 

1)  $w$ is an iterated identity of $\Z\wr \Z$.

2) $w$ is not an identity of $\Z\wr \Z$.

\end {remark}

 Proof of the  remark.
 
1) Note that  $w=(x_1, x_2)=[x_1, [x_1,x_2]]$ belongs to the commutator subgroup of $\Z\wr \Z$ for any $x_1$ and $x_2$.
Observe also that for any $x_1, x_2$ their commutator belongs to the commutator subgroup, and the commutator subgroup of $\Z \wr \Z$ is Abelian.
Therefore, $w(x_1, x_2)=e$ for any $x_2 \in \Z \wr \Z$  $x_1$  and any $x_1$ in the commutators subgroup of $\Z \wr \Z$.
Hence, the second $E$ type iteration of $w$ is trivial.

2)  Let $x, a$ be the standard generators of $\Z \wr \Z$, where $x$ corresponds to the generator of $\Z$ that acts.
Put $x_1=x$, $x_2=x a$. Observe that $[x_1, x_2]$ is not a power of $x$, and this implies that $[x_1,x_2]$ does not commute with $x$.
In other words, $w(x_1,x_2)\ne e$.

qed

The remark implies that $s(\Z\wr \Z) \ge 2$, and hence $s(\Z\wr  \Z) = 2$

\end{proof}

\begin{proposition} \label{propositionzwithfinite}
$s(\Z \wr \Z/R \Z)=k+1$, $k$ being the maximum of $\alpha_i$, where $  R= \prod p_i ^{\alpha_i}$, $p_i$ are prime numbers, $\alpha_i \in \Z_+$.
\end{proposition}

\begin{proof}
We have $G= \Z \wr \Z/n\Z$ and we consider $N= (\Z/R \Z)^\infty \subset G$. $A= G/N = \Z$.
Let $w$ be an iterated identity of $G$.
Since $G$ has elements of infinite order, the total number of $x_i$ in $w$ is equal to zero, for each $i$.
Therefore, $w$ is an identity of $A$.
We consider $u(x_1, x_2, \dots, x_n)$ and $v(x_2, \dots, x_n)$ defined in Remark \ref{remarkuv}. It holds
$w(x_1, \dots, x_n)=u(x_1, x_2 \dots x_n)  v(x_2, x_3, \dots x_n)$.

We want to show that for any $x \in N$ it holds $\phi_{u, \circ k}(x)=0$. Let $\epsilon_i$, $i\in \Z$ be the standard basis of $(\Z/ R \Z)^\Z$.

Suppose that  $\phi_{u, \circ k}(x) \ne 0$ for some $x \in N$. Since $\phi_u$, and hence also all its iterations $\phi_{u, \circ k}$ are homomorphism of the Abelian group
$N$,  exists $i\in \Z$ such that $\phi_{u, \circ k}(\epsilon_i)\ne 0$.
Let
$$
\phi_u(\epsilon_i) = \sum_{i=-M}^M \beta_{j} \epsilon_{i+j} 
$$

By Lemma \ref{commutewithaction} we know that $\phi_u $ commutes with the action of $\Z$ on $(\Z/R \Z)^\Z$ by conjugation. Therefore, for any $r\in \Z$ it holds

$$
\phi_u(\epsilon_r) = \sum_{i=-M}^M \beta_{j} \epsilon_{r+j} 
$$

Observe that if all $b_i$ are divided by any of $p_j$ among the prime divisors of $R$, then all coefficients
 of $\phi_{u, \circ k}(\epsilon_i)$ are divided by
$p_j^k$, for all $j$. In this case all such coefficients are divided by $R$, and thus $\phi_{u, \circ k}(\epsilon_i)=0$ in $(\Z/R\Z)^\Z$.

Therefore, we can assume that there exists some $p$ among prime divisors of $R$ and some integer $j$ such that $\beta_j$ is not divided by $p$.
Let $m$ be the maximal  among $j$ such that  $\beta_j$  is not divided by $p$.
Without loss of generality we can assume 
that $m\ge 0$.

Let 
$$
v=v(x_2,x_3, \dots, x_n) = \sum_{i=-L}^L \gamma_i \epsilon _i
$$

Take any $T >L$. Observe that the $l$-the iteration of $\phi_u v$, applied to $\epsilon_M$, has a coordinate
$\beta_m^l \ne 0$ at the basis vector $\epsilon_{T+ml}$. Therefore, any such iteration is not equal to zero, and hence
no $E$ type iteration of $w$ at $(x_1, \dots, x_n)$ is equal to $e$ in $G$. We have got a contradiction with the fact that $w$ is an iterated identtity of $G$.

Now let us show that $\gamma_i$ is divided by $R$, for any $i$.  Indeed, otherwise there exist a prime divisor $p$ of $R$ such that $R$ is divided by $p^s$, and
there exists $i$ such that $\gamma_i$ is not divided by $p^s$. Let $t$ be the maximal number such that $\gamma_i$ is divided by $p^t$ for all $i$. We know that $t<s$.
Observe that all coefficients of $\phi_{u},  (v)$ are divisible by $p^{t+1}$, and hence   $\phi_{w},  (v) \equiv v$ modulo $p^{t+1}   (\Z/R \Z)^\Z$. This implies that for all $l\ge 1$
$\phi_{w, \circ l},  (v) \equiv v$ modulo $p^{t+1}   \Z/R \Z$. Therefore, $ \phi_{w, \circ l} \ne e$ for all $l\ge 1$. This  contradiction shows that 
$\gamma_i$ is divided by $R$, for any $i$. In other words, $v\equiv 0$ in $(\Z/R \Z)^\Z$.

We have shown therefore that for any $x_2$, $x_3$, \dots, $x_n \in G$ and any $x \in N$ it holds $\phi_{w, \circ k} (x)= w_{\circ k}(x, x_2, \dots, x_n) =e$
This implies that  $s(\Z\wr \Z/R \Z) \le k+1$.

\begin{remark}
Take $w(x_1, x_2)= [x_1^m,x_2]$, where $m=p_1 \dots p_l$, $p_i$ are prime numbers.
Consider $G=\Z \wr \Z/ R\Z$

1)  If  prime divisors of $R$ are contained among  $p_1, \dots, p_l$, then  $w(x_1, x_2)$ is an iterated identity of $\Z\wr \Z/R \Z$

2) If    $n$ is divided by $p_1^k$, then $w_{\circ k}(x_1,x_2)$ is not an identity of $\Z\wr \Z/ R \Z$.

\end {remark}

Proof of the remark.

1)  Observe that $w(x_1, x_2) \in (\Z / R \Z)^\infty$ for any    $x_1, x_2 \in G$
Observe also that for any $x_2 \in G$ and any $k\ge 0$ and $x_1 \in m^k  (\Z / R \Z)^\infty$ it holds
$w(x_1, x_2) \in m^{k+1}  (\Z / n \Z)^\infty$. By the assumption of the remark we know that there exist $k$ such that $m^k$ is divided by $R$. Hence
for any $x_1, x_2 \in G$ the $k+1$-th Engel type iteration of $w$, evaluated at $x_1,x_2$ is trivial.

2)   Let $x, a$ be the standard generators of $\Z \wr \Z R\Z$, where $x$ corresponds to the generator of $\Z$.
Put $x_1=x a$, $x_2=x $. Observe that $[x_1^m,x_2]$ is a configuration of the wreath product that has two non-zero values: $-1$ at $m+1$ and $+1$ at $1$.
Applying $w$ to  $[x_1^m,x_2]$ and $x_2$ we get a configuration, with the value at the maximal point of the support equal to $m$. The $k$-the iteration, applied
to $x_1$ and $x_2$, has value $m^{k-1}$ at one of the points of its support. Since $R$ is divided by $p_1^k$, we know that $m^{k-1} \ne 0$ in $\Z/R\Z$.
We conclude that $k$-iteration of  $w$ is not an identity of $\Z\wr \Z/ R\Z$.

qed

The remark implies that $s(\Z\wr \Z/R\Z) \ge k+1$ for $R$, $k$ satisfying  the assumption of the theorem).

\end{proof}

\subsection{Central extensions}

\begin{remark} \label{kernelnormalsubgroup}

 Let $H$ be a normal subgroup of $G$, $w(x_1, \dots, x_n)$ be some word and denote by  $\phi_d$ the verbal map $\phi_d(x)= w_{\circ d}(x, x_2, \dots, x_n)$.

i) For any $x_2, \dots, x_n$  there exists $D$ (an integer number or $\infty$) such that the following holds. For any $x \in H$  the image $\phi_l(x) \in N$ if and only if $l$ is divided by $D$. 

ii) If we assume additionally that $w$ is an iterated identity of  $G/H$, then $D \le s(w, G/H)+1$.

\end{remark}

\begin{proof}

i) Apply Remark \ref{remarkuv} to $w_{\circ d}$. We know that there exist words $u_d$ and $v_d$ such that $\phi_d (x)= \phi_{u_d}(x, x_2, \dots, x_n) v_d(x_2, \dots, x_n)$.
Observe that  $ \phi_{u_d}(x, x_2, \dots, x_n)$ is a product of conjugates of powers of $x$. Therefore  $\phi_{u_d}(x, x_2, \dots, x_n)$ belongs to $H$ whenever
$x\in H$, since $H$ is a normal subgroup.
If $v_d(x_2, \dots, x_n)$ belongs to $H$ then $\phi_d (x) \in H$ for any $x \in H$.
If $v_d(x_2, \dots, x_n)$  does not belong to $H$ then $\phi_d (x) \notin H$ for any $x \in H$.

Now observe that  if $D$ is such that $\phi_{\circ d}(x) \in N$ for any $x\in N$, then  $\phi_{\circ kD}$ also has property. Observe also that if $a>b$ and $\phi_{\circ a}$ and
$\phi_{circ b}$ both have this property, then $\phi_{\circ (a-b)}$ also has this property.

ii) If $\phi_1 (x)  \in H$ for some $x\in H$, then $D$ (defined in i) ) is equal to $1$, and thus $1=D \le s(w, G/H)$.
Otherwise, $\phi_1(e) \notin H$. In this case there exists $l \le s(w, G/H)$ such that $\phi_{l+1} = \phi_l (\phi(e)) \in H$, and thus
$D \le s(w, G/H)+1$.

\end{proof}

\begin{lemma}  \label{lemmacentral}
 Assume that $H$ is a central subgroup of $G$. For any iterated identity $w$  of $G$ such that the total number of each $x_i$ in $w$ is equal to zero
$$
s(w,G) \ge 2 s (w, G/H) +1.
$$
In particular, if $G$ contains at least one element of infinite order and 
if $G/H$ is bounded, then $G$ is bounded.

\end{lemma}

\begin{proof}

Let $d = s(w, G/H)$. Fix some $x_2$, \dots $x_n \in G$.

Let $\phi_d$ is a verbal map $\phi_d(x)= w_{\circ d}(x, x_2, \dots, x_n)$.

By Remark \ref{kernelnormalsubgroup} we know that there exists $D \le d+1$ such that for any $x \in N$  the image $\phi_l(x) \in N$ if and only if $l$ is divided by $D$. 

Apply Remark \ref{remarkuv}  to $w_{\circ D}$ and write $\phi_D(x)= u(x, \dots, x_n) v (x_2, \dots, x_n)$.
Since the total number of entries of  $x_1$ in $w$ is equal to zero and since $H$ is central, we see that  $u(x, \dots, x_n) =e$.
Let us show that $v(x_2, \dots, x_n)=e$. Indeed, otherwise $w_{\circ (lD)}(e, x_2, \dots, x_n) = v(x_2, \dots, x_n) \ne r$ for any positive integer $l$.
This implies that  $w_{\circ m}(e, x_2, \dots, x_n) \ne e$ for any positive integer $m$, and this is a contradiction with the fact that $w$ is an iterated identity of $G$.

We have shown therefore that $v(x_2, \dots, x_n)=e$, and thus $w_{\circ D} (x, x_2, \dots, x_n)=e$ for any $x \in H$.

This implies that for any $x_2, x_3, \dots, x_n$ and any $x\in G$ there exists $l \le d+D \le 2d+1$ such that
$w{\circ l}(x, x_2, \dots, x_n) = e$, and hence $s(w,G) \ge 2 s (w, G/H) +1$.

Finally observe that if $G/H$ satisfies some non-trivial identity, so does $G$: if $G$ satisfies an $E$ type iterated identity $w(x_1, \dots, x_n)$, then
$G/H$ satisfies the commutator identity $w'(x_1, \dots, x_{2n})= [w(x_1, \dots, x_n), w(x_{n+1}, \dots, x_{2n})]$.
\end{proof}

\begin{example}
Let $G$ be a solvable group such that the center $C(G) =  \sum_{i=1}^\infty \Z/p^i \Z$ and such that  $G/C(G) = \Z \wr \Z$.
(Such groups do exist, see Section 3 in \cite{hall1954}).
Then $G$ is bounded,, but $C(G)$ is not bounded.
\end{example}

\begin{proof}
From Proposition \ref{propositionwithz} we know that $\Z \wr \Z$ is bounded.  Since $G$ is a central extension of $\Z \wr \Z$, we can use
 Lemma \ref{lemmacentral} to conclude that $G$ is bounded.

\end{proof}

\section{Metabelian groups} \label{sectionmetabelian}

\begin{theorem} \label{theoremmetabelian}
   If $G$ is a finitely generated metabelian group,
 then $s(G)$ if finite.
\end{theorem}

\begin{proof}

If $G$ is finite, we know that $G$ is bounded (see Example \ref{examplefinite}). We will therefore assume that $G$ is infinite.

Let $N$ be an Abelian normal subgroup of  $G$, such that $A=G/N$ is Abelian.

Let $w$ be an iterated $E$-type identity of $G$. Since $G$ is infinite finitely generated metabelian group, it contains an element of infinite order. Therefore, for each $i$ the total number of $x_i$ in $w$ is equal to zero. In particular, $w$ is an identity of $A$.  
Using Remark \ref{remarkuv} we can write 
$w(x_1, \dots, x_n)=u(x_1, \dots, x_n)  v(x_2, \dots, x_n)$.

Fix $x_2, \dots, x_n$ and consider the homomorphism $\phi_u : N  \to N$.
(By Lemma \ref{lemmahomom} we know that $\phi_u$ is indeed a homomorphism). We also consider $\phi_w(x)= \phi_u(x)v(x_2, \dots, x_n)= w(x, x_2, \dots, x_n)$.

 Since $A$ is a finitely generated Abelian group,
$A=\Z^d +A_f$, $d\ge 0$, $A_f$ is a finite Abelian group.

Given a map $\rho$ from $N$ to $N$ let us say that $\rho$ is  recurrent if for any $x\in N$ there exists $n$ such that the $n$-th iteration $\rho_{\circ n}(x)=e$.
Since $w$ is an iterated identity, we know that the map $\phi_w: N \to N$ is recurrent. Our goal is to find a common upper bound for (minimal) $n$ such that
$\phi_{w, \circ n}(x)=e$, which does not depend on $x_2, \dots, x_n$ and does not depend on $w$. To do this, we need to study dynamics of the homomorphism $\phi_u$ and 
of the "affine map" $\phi_w$ on $N$.

\begin{lemma} \label{lemmafinitetensor}
Let $F$ be a torsion-free Abelian group such that its tensor product $F \otimes \mathbb{Q}$ has finite dimension, which we denote by $D$.
 $\xi: F \to F$ be a homomorphism, $a$ is some element of $F$  and
$\rho(x) = \xi(x)+a$. 
Suppose that for any $x\in F$ there exists $n$ such that $\rho_{\circ n} (x)=e$. Then for any $x\in F$  it holds
 $\rho_{\circ D} (x)= \xi_{\circ D}(x)=e$.

\end{lemma}

\begin{proof}
Since $F$ is torsion free, $F$ is imbedded in  $V=F \otimes \mathbb{Q}$.

 It is clear that $\xi$ extends to a linear map from $V$ to $V$, and we use the same letter $\xi$ to denote this
linear map.
First let us show that the kernel of $\xi_{\circ D}$ is equal to $V$. 
Let $V_i = \ker \xi_{\circ i}$.
It is clear that $V_1 \subset V_2 \subset V_3 \dots$.
It is also clear that if for some $j$ it holds $V_j=V_{j+1}$, then for any $k\ge j$ it holds $V_k=V_j$.
Suppose that the kernel of $\xi_{\circ d}$ is not equal to $V$. Then there exists a linear subspace $U$ of dimension $\le D-1$ such that for any $i\ge D-1$ it holds
$V_i=U$.

Now observe that there exits $M$ such that $\phi_{\circ M}(0)=\phi_{\circ (M-1)} (\phi (0))=0$. Note that
$\rho$ (as well as its iterations) extends to an affine map from $V$ to $V$, denote by
$W_i$ the kernel of $\phi_{\circ i}$ in $V$.

Observe that that $W_i$ is contained  in the union of the image of $U$ by a finite number of translations.
In other words, the union $W = \cup W_i$ is contained in $\cup_{i=1}^l (U+ v_i)$, for some $l\ge 0$ and some $v_i \in V$.
Note that since the linear span of the image of $F$ is equal to $V \ne U$, there exists a $f \in F$ such that its image does not belong to $U$.
Replacing, if necessary, $f$ by some of its power $Lf$, we conclude that there exists $f \in F$ such that its image in the tensor product
 does not belong to $W$.
This is a contradiction with the assumption of the lemma.

Now we know that  the kernel of $\xi_{\circ D}$ is equal to $V$: $V_D = V_{D+1}= V_{D+2}= \dots =0$ and $\xi_{\circ D}(x)=0$ for any $x\in F$.

Observe that $\phi_{\circ j}(0) =  a + \xi(a)+ \xi_{\circ 2}(a) + \dots + \xi_{\circ (j-1)}(a)$. This shows that
for any $j\ge D$ it holds $\phi_{\circ j}(0) = \phi_{\circ d}(0)$. Since $\phi_{\circ M}(0) =0$ for some $M\ge 1$,  and thus we conclude that 
$\phi_{\circ j}(0)=0$ for any $j\ge D$.

This implies that $\phi_{\circ j}(x)= \xi_{\circ j}(x)$ for any $j\ge D$ and any $x\in F$, and completes the proof of the lemma.

\end{proof}

In Lemma \ref{lemmaalongcoordinates} below we use additive notation for $N$ and $G/N$.

\begin{lemma}  \label{lemmaalongcoordinates}

 Let $G/N$ and $N$ be Abelian groups.  Let $w(x_1, \dots, x_n)$ be an identity of $G/N$ and $u$ and $v$ be as in Remark \ref{remarkuv},
Assume that $G/N= \mathbb{Z}^d+A'$, $A'$ is finite.
Denote by $\epsilon_1, \dots, \epsilon_d$ the standard generators of $\Z^d$.

i) There exists $l_j$ ( $j\in J$, $J$ is a finite set), and $r_j=(r_{2, j}, \dots, r_{d})$ 
 such that  
for any $x_2, \dots, x_n \in G$ and $x \in N$ 
$$
\phi_ u(x)= \sum_{j\in J, z= ({r_{2,j} x_2},  {r_{3,j} x_3}, \dots, {r_{n,j} x_n}        )} l_j z(s)
$$

and such that for any $k: 1 \le k \le d$ there exists $x_2, \dots, x_n \in G$ such that for any $x\in N$
$$
\phi_ u(x)= \sum_{j\in J} l_j x^{M_j \epsilon_ k},
$$  
where $M_j$ are integers such that $M_j \ne M_{j'}$ for $j, j' \in J$ such that $j\ne j'$.

In particular, if  there exists $x_2, \dots, x_n$ such that the homomorphism $\phi_u$ from $N$ to $N$ is not identically zero, then there exists $x_2, \dots, x_n$ such that
$\phi_u$ has the form as above and at least one of coefficients  $l_j$ is non-zero.

ii) Assume moreover that $w$ is an iterated identity of $G$ and that $N$ contains at least one non-identity element. If there exists $x_2, \dots, x_n$ such that the homomorphism $\phi_u$ from $N$ to $N$ is not identically zero, then there exists $x_2, \dots, x_n$ such that
$\phi_u$ has the form as above and at least two of coefficients  $l_j$ is non-zero.

\end{lemma}

\begin{proof}

We know that 
$w(x_1, \dots, x_n)= u(x_1, \dots, x_n) v(x_2, \dots, x_n)$, $u$ and $v$ are identities of $G/N$. Fix some $x_2$, \dots, $x_n$ in $G$. Put $\psi(x)=
\phi_u(x)= u(x, x_2, \dots , x_n)$.
For any $ x\in G$
$$
u(x, x_2, \dots, x_n) = \prod \alpha_i(x_2, \dots, x_n) x^{l_i} \alpha_i(x_2, \dots, x_n) 
$$

Observe that for any $x\in N$ $\alpha_i(x_2, \dots, x_n) x^{l_i} \alpha_i(x_2, \dots, x_n) $ depends only on the image of $\alpha_i$ in $G/N$.
Regrouping, if necessary, the terms with the same image of $\alpha_i$ in $G/N$, we conclude that for any $x_2, \dots, x_n \in G$ and $x\in X$
$\phi_u$ has the following form, in the additive notation on $N$ and $G/N$:

$$
\phi_ u(x)= \sum_{j\in J, z= ({r_{2,j} x_2},  {r_{3,j} x_3}, \dots, {r_{n,j} x_n}        )} l_j (s)^z
$$

Now we fix $k: 1 \le k \le d$ and consider $x_2, \dots, x_n \in G/N$ (or to be more precise, consider $x_2, \dots, x_n$ in $G$ with corresponding image  in $G/N$
$x_2 =L_2 \epsilon_k$, $x_3 =L_2 \epsilon_k$,  $x_n =L_n \epsilon_k$.
Consider the lexicographic order on $\Z^{n-1}$.
Observe that if  positive integers $L_i$ are such that $L_2$ is much larger than $L_3$, $L_3$ is much larger than $L_4$, \dots( $ L_2>> L_3 >> L_4 >> \dots >> L_n$) and if
$r= (r_2, \dots, r_n)> \bar{r}=  (\bar{r}_2, \dots,\bar{r}_n) $ in the lexicographic order on $\Z^{n-1}$, then 
$$
L_2 r_2 + L_3 r_3 + \dots + L_n r_n >  L_2 \bar{r}_2 + L_3 \bar{r}_3 + \dots + L_n \bar{r}_n
$$
Thus, choosing a quickly decreasing sequence of positive integers $L_2, \dots, L_n$, we can assure that
all $M_j =L_2 r_{2,j} + L_3 r_{3,j} + \dots + L_n r_{n, j}$, $j$ belonging to the finite set $J$, are distinct.

This completes the proof of i).

Two prove ii)  take some $x\ne e, x \in N$.
Observe that if $\phi_ u(x)= \sum_{j\in J} l_j x^{M_j \epsilon_ k}$ for some choice of $x_2, \dots, x_n$ and all $x$ and exactly one among coefficients $l_j$ is non-zero, then any iteration of  $\phi_ u$ has the same
property, which contradicts the fact that some iteration of $\phi_ u(x)$ is trivial.

\end{proof}

\begin{lemma} \label{lemmacubeisfinite}

Let $H$ be an Abelian group, $A=\Z^d$ act on $H$ and there exists $s\in H$ such that $H$ is generated by $a(s)$, $a\in A$.
Let $\epsilon_1, \dots, \epsilon_n$ denote the standard  generators of $\Z^d$.
Suppose that for any $i: 1 \le i \le d$ there exists $k>0$ and  integers $m^{(i)}_1< m^{(i)}_2< \dots m^{(i)}_k$  and $\gamma^{(i})_1, \gamma^{(i)}_2, \dots, \gamma^{(i)}_{k_i}$ such that
(for each $i$) at least two among $\gamma^{(i)}_j$ are non-zero and
$$
\sum_j \gamma^{(i)}_j  s^{(m^{(i)}_j \epsilon_i)}=0. 
$$
Then $F \otimes \mathbb{Q}$ has finite dimension.

\end{lemma}

\begin{proof} Denote by $\pi$ the tensor map from $F$ to $F \otimes \Q$.
For any $i : 1 \le i \le d$ chose $m^{(i)}_j$, $\gamma^{(i)}_j$ as in the assumption of the lemma. Without loss of generality we can assume that $\gamma^{(i)}_k\ne 0$, $\gamma^{(i)}_1 \ne 0$ for each $i$.
We have therefore
$$
s^{m^{(i)}_{k_i} \epsilon_i } = - \frac{1}{\gamma^{(i)}_{k_i}} \sum_{j=1}^{k_i-1} s^{(m^{(i)}_j \epsilon_i)}    \mbox{        and     }         s^{m^{(i)}_1 \epsilon_i} = - \frac{1}{\gamma^{(i)}_1} \sum_{j=2}^{k_i} s^{ (m^{(i)}_j \epsilon_i) }.
$$

This implies that for any $a \in \Z^d$
$$
s^{(m_{k_i} \epsilon_i  +z)} = - \frac{1}{\gamma_{k_i}} \sum_{j=1}^{k_i-1} s^{ (m^{(i)}_j \epsilon_i+z) }
\mbox{    and   }
s^{(m^{(i)}_1 \epsilon_i +z)} = - \frac{1}{\gamma_{k_i}} \sum_{j=2}^{k_i} s^{(m_j \epsilon_i +z)}.
$$

Consider the finite cuboid $P \subset \Z^d$ containing $(z_1, \dots, z_d):   m^{(i)}_1 \le z_i \le m^{(i)}_{k_i}$.
Consider the  finite set  $p(s)$, $p \in P$, and denote this finite set by $F_P$.
It is clear that the linear span of the image of $F_P$ for the tensor mapping is equal to $F \otimes \mathbb{Q}$, and hence
$F \otimes \mathbb{Q}$ is finite dimensional.

\end{proof}

Let us say that a map $\phi: X \to X$ is  weakly recurrent, if there exists a finite set $X_f \subset X$ such that for any $x$ there exists
$n\ge 1$ such that  the $n$-th iteration $\phi_{\circ n}(x) \in X_f$.

\begin{remark} If $\phi$i is recurrent (with respect to some fixed point $e \in X$), and if $d \ge 1$, then $\phi_{\circ d}$ is weakly recurrent.

\end{remark}

\begin{proof}
Put $X_f = \{  e, \phi(e), \phi_{\circ 2}(e), \dots, \phi_{\circ d}(e)    \}$. It is clear, that if $\phi$ is recurrent, than for any $x$ its iteration $\phi_{\circ d}$ visits the set
$X_f$ (infinitely many times).
\end{proof}

We return to the proof of the Theorem. We have some $x_2, \dots, x_n$ fixed, and we consider the dynamics of $\phi_u$ and $\phi_w$.

Let $T$ be the torsion subgroup of $N$. $F/T$ is torsion free.

The homomorphism $u$ preserves $T$ and thus induces the quotient homomorphism $\phi_u^{(F/T)}: F/T \to F/T$.
Let $v'$ is the image of $v(x_2, \dots, x_n)$ in $F/T$.
It is clear that
$$
\pi_{F/T} (  \phi_w(x))  =  \phi_u^{(F/T)}(\pi_{F/T}(x)) +v'.  
$$

Since finitely generated metabelian groups satisfy maximal condition for normal subgroup (Hall,  \cite{hall1954}), we know that $N$ is a normal closure of a finite set $S'$. 
This implies that  $N$ is generated by $ s^a$, $s \in S'$, $a \in A$ (where $s^a$ denotes the action of $a$ on $s$ by  conjugation). Considering the orbit of $S$ under the action of
the finite group $A_f$, we conclude that there exists a finite set $S \subset N$ such that $N$ is generated by $ s^a$, $s \in S$, $a \in \Z^d$

Let $T$ be the torsion subgroup of $N$. Observe that $T$ is a normal subgroup of $G$. Indeed, its conjugation by an element of $G$ belongs to $N$ since $N$ is normal, 
and a conjugation of a torsion element is a torsion element.

Therefore, using once again  \cite{hall1954} we conclude  that there exists a finite set $S'_T$ such that $T$ is generated as a normal subgroup by $S'_T$. In other words, $T$ is generated by $g s g^{-1}$, where
$g\in G$, $s\in S'_T$ and generated by $a s a^{-1}$, where
$a\in A$, $s\in S'_T$.
Considering, as before, the orbit of $S'_T$ under the action of the finite group $A_f$, we conclude that there exists a finite set $S_T \subset T$ such that
$T$ is generated by $a s a^{-1}$, where
$a\in \Z^d$, $s\in S_T$.

For any $s\in S$ consider the linear span  $V_s$ of  the images $a(s) , a \in \Z^d$ in $F \otimes \mathbb{Q}$. Let $d_s$ denotes the dimension of $V_s$.
Recall that $S$ is a finite set. Let $d$ be the maximum among finite values of $d_s$, $s \in S$. If all $d_s$ are infinite, we put $d =0$.
Let us show that $\phi^{F/T}_{u, \circ d+1)}(x)=0$ and $\phi^{F/T}_{w, \circ d+1}(x)=0$ for any $x \in F/T$.

Observe that since $\phi_u$ commutes with the action of  $A$ and $\phi_u(s)$ belongs to the subgroup generated by $a(s)$, $a\in \Z^d$, it suffices to consider the
action induces bt $\phi_u$ on each $V_s$. Observe that if  $s\in S$ is such that the dimentsion  $d_s$ of $V_s$ is  $< \infty$, then for any $d > d_s$  $\phi_{u, \circ d}^{F/T}(t), \phi_{w, \circ d}^{F/T}(t) = 0$ for any $t \in V_s \cap \pi(F)$, and hence also for any $ t \in V_s$
(follows from Lemma \ref{lemmafinitetensor}).
Otherwise, if  $s$ is such that $d_s=\infty$, observe that  ii) of Lemma \ref{lemmaalongcoordinates} and Lemma \ref{lemmacubeisfinite} imply that the restriction of 
$\phi_u^{(F/T)}$ and $\phi_w^{(F/T)}$ to  $V_s \cap \pi(F)$  (and hence the restrictions to $V_s$ as well) are identitcally zero.

We have shown therefore that $\phi^{F/T}_{u, \circ d+1)}(x)=0$ and $\phi^{F/T}_{w, \circ d+1}(x)=0$ for any $x \in F/T$.

We implies  that $\phi_{w, \circ d+1}(x) \in T$ for any $x \in F$.
Now we consider $\bar{w}= w_{\circ d}$, chose $\bar{u}(x, x_2, \dots, x_n)$, $\bar{v}(x_2, \dots, x_n) \in T$ and consider the dynamics
of the homomorphism $\phi_{\bar{u}}$ on $T$.

\begin{lemma} \label{lemmadimfinitefinite}

$H$ is an Abelian group such that $ m^k h = 0$ for all $h \in H$,  where $m = p_1 \dots p_l$, $p_i$ are prime.
Suppose that $H \otimes \Z/p_i\Z$  (viewed as a vector space over $\Z/p_i \Z$) is finitely dimensional for all $i$. Then $H$ is finite.

\end{lemma}

\begin{proof}  
First suppose that $l=1$, $H$ is an abelian group such that $p^k h =0$ for all $h$. $H \otimes \Z/pZ$ is finitely dimensional.
Consider $H_i= p^i H$. It is clear that $H_i$ is a subgroup of $H$ and $ H=H_0 \supset H_1 \supset \dots \supset H_k =0$.
Observe that $H_0 / H_1 = H/p H = H \otimes \Z/pZ$. By the assumption of the Lemma, this is a finitely dimensional vector space over a finite fields,
therefore $H_0/H_1$ is a finite group.
Observe that $H_i/ H_{i+1}$ is a quotient of $H_0/ H_1$ by the map which is a multiplication by $p^i$, and hence $H_i/H_{i+1}$ is finite for any $i$.
This implies that $H$ is finite.

Now consider the general case. Prove the statement by induction on $l$. The statement of the base of the induction $l=1$ is already proven.
Consider $H' = p_l^k H$. Observe that $H/H'$ satisfies the assumption of the induction step for $l-1$, and hence $H/H'$ is finite. Observe also that $H'$ satisfies
the assumption of the statement of the induction base, and thus $H'$ is finite. We conclude that $H$ is finite.

\end{proof}

We want to show that there exists a finite subgroup $T_f \subset T$ such that  $\phi_{\bar{u}, \circ k}(x)=0$ for any $x \in T$. Here $k$ is such that $m^k t =0$ for any $x \in T$ and $m$ is a product of (finitely many) distinct prime numbers.

Indeed, 
take $t \in T$ and consider $T_t$ to be the linear span of $t^a$, $ a\in A$ and $T'_t$ to be the linear span of $t^a$ $a  \in \Z^d$.
It is clear that $T_t$ is  equal to 
the linear span of $(T'_t)^a= T'_{t^a}$, $a \in A_f$. Observe $T'_{t^a}$ is isomorphic to $T'_t$ for any  $a$,     and hence 
the dimension of $T'_{t^a} \otimes \Z/p\Z$ is equal to 
that of  $T'_{t} \otimes \Z/p\Z$ (and the same statement for the tensor product with $\mathbb{Q}$.)

If $t$ is such that the dimension of $T_t \otimes \Z/pZ$ is infinite, then the dimension of $T'_{t} \otimes  \Z/p\Z$ is infinite and also  $T'_{t^a} \otimes  \Z/p\Z$ is infinite for
any $a \in A_f$. In this case   Lemma  \ref{lemmaalongcoordinates} and Lemma \ref{lemmacubeisfinite} imply
that all but possibly one coefficient of $\phi_u$ are divided by $p$. 
This implies that all such coefficients is divided by $p$, since otherwise there exist $x_2, x_3, \dots, x_n$ such that exactly one coefficients among $l_j$, defined in 
Lemma  \ref{lemmaalongcoordinates} is not divided by $p$. Observe that there exists $x \in T$ such that $x \notin pT$, and in this case any iteration of $\phi_u$, applied to $x$, does not belong to $p T$, and this is in contradiction with the fact that one of such iterations is trivial.

Hence $\phi_{u, \circ k}(x) $, as well as  $\phi_{w, \circ k}(x)$ belong
to $m_t^k T$, where $m_t$ is the product of $p_i$, such that $p_i$ is a prime divisor of $m$ and dimension of  $T_t \otimes \Z/p_i\Z$ is infinite.

Now note that for any $p_i$ the dimension of $m_t^k T \otimes \Z/p_i \Z$ is finite. Lemma \ref{lemmadimfinitefinite} implies that the subgroup $m^k T$ is finite.
Put $T_f$ be the subgroup generated by $m_t^k T_t$, $t$ is in the finite set $S_T$ ( the image of which under the action of $A$ generates $T$). $T_f$ is an Abelian group generated by a finite number of Abelian groups, and hence $T_f$ is finite. Observe that  $\phi_{u, \circ k}(t)$, as well as  $\phi_{w, \circ k}(t)$ belongs to $T_f$ for any $t\in T$.

We know therefore that there exists a finite group $T_f$ and  $C= D+k$ such that for any iterated identity $w$ of $G$ it holds $w_{\circ C}(x_1, \dots, x_n)$ belongs to
$T_f$ for any $x_1, \dots, x_n \in T_f$.

Now fix $w$ and $x_2$, \dots $x_n \in G$.
Let $B_w \subset T_f$ consists of the elements  $z \in T_f$ such that  $w_{\circ C'}{e, x_2, \dots, x_n} \in T_f$ for some $C: 0 <C'' \le C$. 
Since $w$ is  an iterated identity, for any $x \in G$ there exists $m$ such that $w_{\circ C m}(x, x_2, \dots, x_n) \in B_w$. Any orbit of the word map $w_{\circ C}(*, x_2, x_3, \dots, x_n)$ has size
at most $\# T_f$, and therefore for any $x \in G$ there exists $C_2 \le  \# T_f $ such that  $w_{\circ C C_2}(x, x_2, \dots, x_n) \in B_w$. This implies that there exists $C_3 \le C (\# T_f+1)$
such that $w_{\circ} (x, x_2, \dots, x_n)=e$.

We have shown therefore that for any iterated identity $w$ of $G$ its iterational depth $s(w,G) \le  (D+k) (\# T_f+1)$.

This completes the proof of the theorem.

\end{proof}

\section{$S$ type iterated identities } \label{sectionfurthernotions}

\begin{definition} \label{definitionSolvability}
We say that a group $G$ satisfies $S$-type (or solvability type) iterated identity $w$ if
for any $x_1, \dots, x_n \in G$ there exists $N$ such that for any $y_1, \dots y_{n^N}$ taking
 value among $x_1, \dots, x_n$
it holds
$$
w_{\star N}(y_1, y_2, \dots, y_{n^N} )=   w(w( w(\dots w(y_1, y_2, \dots y_n), w(y_{n+1}, y_{n+2}, \dots  y_{2n}) \dots ), \dots ))=e.
$$

\end{definition}

\begin{example} \label{exampleSsolvable}
 $w_0(x_1, x_2)=x_1x_2x_1^{-1}x_2^{-1}$. 
If $G$ is a solvable group, then $G$ satisfies $S$-type identity $w_0$.

More generally, let $w(x_1, \dots, x_n)$ be such that the total number of each $x_i$ in $w$ is equal to zero.
Then any solvable group $G$ satisfies the $S$ type identity $w$.
\end{example}

\begin{proof}
Let $G^{(n)}$ denotes the $n$-th group in the derived series of $G$, that is $G_1=[G,G]$ and $G_{i+1} = [G_i, G_i]$ for any $i \ge 1$.
Observe that $w$ is an identity in any Abelian group, and hence $w(x_1, \dots, x_n) \in [G,G]$ for any $x_1, \dots, x_n \in G$. Moreover, 
if $x_1, \dots, x_n \in G^{(i)}$, then $w(x_1, \dots, x_n) in G^{(i+1)}$. Since $G$ is solvable, there exists $N$ such that $G^{(N)}=N$.
Put $k=n^N$. It is clear that  $w_{\star N}(y_1, \dots, y_k) =0$
for any $y_1, \dots, y_i \in G$, 
\end{proof}

\begin{example} \label{exampleScyclic}
$\Z$ satisfies an $S$ type iterated identity $w(x_1, \dots, x_n)$ if and only if the total number of each $x_i$ in $w$ is equal to zero.
\end{example}

\begin{proof}
Suppose that the total number of some $x_i$ in $w$ is equal to $m \ne 0$. Without loss of generality we can assume that $i=1$.
Let $1$ denotes  a generator of $\Z$.
Consider  a finite set consisting of $0$ and $1$.
For each $N$ put $k=2^N$ and consider $y_1=a$, $y_2= y_3 = \dots y_k=0$.
Observe that $w_{\star N}(y_1, \dots, y_n)= m^N \ne 0$, and hence $w$ is not an $S$ type iterated identity of $G$.

Observe also, that if the total sum of each $x_i$ in $w$ is equal to zero, then $w$ is an identity of $\Z$. This implies that
$w$ is an $S$ type iterated identity of $\Z$.
\end{proof}

\begin{example}
Let $G$ be a solvable group containing at least one element of infinite order.
Then $G$  satisfies an $S$ type iterated identity $w(x_1, \dots, x_n)$ if and only if the total number of each $x_i$ in $w$ is equal to zero.
\end{example}

\begin{proof}
Follows from Examples \ref{exampleSsolvable} and \ref{exampleScyclic}.
\end{proof}

Similarly to $E$ type iterated identities, the property to satisfy $S$ type iterated identity is  closed under taking subgroups, quotients, union of increasing group sequences, and it addition this property
is closed under taking extensions:

\begin{lemma} \label{lemmasolv} \label{lemmaECStype}
\begin{enumerate}

\item If $G$ satisfies an $S$-type iterated identity $w$, then
any subgroup of $G$ has the same property.

\item If $G$ satisfies an $S$-type iterated identity $w$, then
any quotient of $G$ has the same property.

\item  If $G_1 \subset G_2 \subset G_3 \dots$, and each $G_i$ satisfies an $S$-type identity $w$, then the union
$G= \cup_{i=1}^\infty G_i$   satisfies the same property.

\item If $H$ is a normal subgroup of $G$, $H'=G/H$, $H,H'$ satisfy an $S$ type iterated identity $w$, then $G$ has the same property

\end{enumerate}
\end{lemma}

\begin{proof}
(1)-(3) are straightforward.

To prove (4), observe that for any since $G/H$ satisfies $S$-type iterated identity $w$,  for any $n_1$ there exists $N_1$ such that for any
 $t_1, \dots, t_{n_1} \in G$ and any
 
 $y_1, \dots, y_{k_1} $, $k_1=n_1^{N_1}$    taking values among  $t_1, \dots, x_{t_1}$  

$$
w_{\star N_1}(y_1, \dots, y_{k_1}) =   w(w( w(\dots w(y_1, y_2, \dots y_{n_1}), w(y_{n_1+1}, y_{n_1+2}, \dots  y_{2n_1}) \dots ), \dots )) \in H.
$$

Consider a set $V \subset H$ which consists of possible values of $w_{\star k}{y_1, \dots, y_{k_1}}$, where  $k_1=n_1^{N_1}$ and all $y_i$ take value
among  $x_1, \dots, x_{n_1}$. It is clear that $V$ is finite,  since its cardinality, which we denote by $n_2$, is at most $n_1^{k_1}$.

Since $H$ satisfies the $S$ type iterated identity $w$, there exists $N_2$ (depending on $V$) such that for $k_2=n_2^{N_2}$ and any
$z_1, z_2, \dots, z_{k_2}$ taking values  in $V$ it holds

$$
w_{\star N_2}(z_1, \dots, z_{k_2}) =   w(w( w(\dots w(z_1, z_2, \dots z_{n_2}), w(z_{n_2+1}, z_{n_2+2}, \dots  z_{2n_2}) \dots ), \dots )) =e.
$$

Put $N= N_1+N_2$.  Consider $k=n_1^N$. Observe,  that for any $r_1, r_2, \dots, r_k$ taking values among  $t_1, \dots, t_{n_1}$  it holds
$$
w_{\star N}(r_1, \dots, r_{k}) = e,
$$
and this concludes the proof of 4).

\end{proof}

Given a class of groups $C$, denote by $E(C)$ the minimal class of groups, which is closed under taking four elementary operations, that is, taking subgroup, 
quotient, direct limit and group extensions.

\begin{corollary}
Let $w$ be a word
Let $C_{w,N}$ be the variety of groups, generated by the identity $w_{\star N}$, that is, the class of groups that satisfy the identity $w_{\star N}$. Put $C_w= \cup_{N=1}^\infty C_{w, N}$.
  The class of groups
satisfying an $S$ type iterated identity $w$ contains $E(C_w)$.
\end{corollary}

 In general, this class is larger than $E(C_w)$.

\begin{example} \label{grigorchuknotelementary}
Let $C$ be a class of groups of bounded $2$-torsion.  We know that the fist Grigorchuk group $G_{\rm Grig}$ satisfies Solvability type (as well as Engel type) iterated identity $w(x_1)=x_1^2$. However,
$G_{\rm Grig}$ does not belong to $E(C_w)$.
\end{example}

\begin{proof}

If the class $C$ is closed with respect to taking subgroups and quotients, than it is not difficult to see that the elementary class $E(C)$ is the minimal class of groups, containing $C$ and closed
with respect to taking direct limits and extensions by a group in $C$ (see Corollary 2.1 \cite{osinelementaryclasses}).

More precisely, given an ordinal $\alpha$, denote by $E_\alpha(C)$ in the following way. Suppose we have defined $E_\beta(C)$ for any $\beta <\alpha$.
If $\alpha$ is a limit ordinal, then $E_\alpha(C)$ is defined as a union $\cup_{\beta<\alpha} (C)$. If $\alpha$ is not a limit ordinal, then $E_\alpha(C)$ is the class of groups containing direct limits of groups in $E_{\alpha-1}$ and extentsions of groups in $E_{\alpha-1}$  by a group of $C$. Then for each $E_{\alpha-1}$ , for any $\alpha$,  is closed with respect to taking subgroups and quotients (see Lemma 3.1 in \cite{osinelementaryclasses}).

Let $C=C_w$. It is clear that $C$ is the class of groups of bounded $2$-torsion and that $C$ is closed with respect to takings subgroups and quotients.

Consider a minimal ordinal $\alpha$ such that $E_\alpha$ contains a $G_{Grig}$.
It is clear that $\alpha$ is not a limit ordinal, and that $G$ can be obtained by taking a quotient from the groups in $E_{\alpha-1}$
(since $G$ is finitely generated, the last operation to obtain $G$ could not be a direct limit).
Therefore, there exists a normal subgroup $N$ in $G_{\rm Grig}$ such that  $N$ and $G/N$ belongs to $E_{\alpha-1}$. Since $G \notin E_{\alpha-1}$, $N \ne {e}$.

Any proper quotient of  $G_{Grig}$ is finite,
 $G/N$ is finite, $N$ is a normal subgroup of finite index in $G$. Any non trivial normal subgroup in 
 $G_{Grig}$ is of finite index and contains some congruence subgroup $St(n)$ (Grigorchuk, \cite{grigorchuk200 }[Proposition 10]).
 This implies that $St(n)$ belongs to $E_{\alpha-1}$. Since any congruence subgroup contains $G_{Grig}$ as a quotient (taking restriction on the first branch of level
 $n$ of the tree, \cite{grigorchukparabolic}), this also implies that $G_{Grig}$ belongs to $E_{\alpha-1}$. This is a contradiction that $\alpha$ is a minimal 
 ordinal such that $E_\alpha$ contains  $G_{Grig}$.

\end{proof}

\subsection{Quasi-varieties} 

Recall that a variety of groups is a class of groups for which a certain set  of identities hold.
This class is obviously closed under taking subgroups, homomorphic images (the same is to say: by taking quotients), and direct products.

Given a set $\Omega$ of iterated identities, we consider the class of  groups $G$ such that $G$ satisfies an $E$ type (correspondingly $S$ type)
 iterated identity $w$, for all $w\in \Omega$.
We call this class of groups an Engel type (correspondingly $S$ type) {\it quasi-variety} generated by a set of iterated identities $\Omega$.

\begin{example} Two nilpotent groups $N_1$ and $N_2$ generated the same $E$-type iterated quasi-variety if and only if they generated the same variety of groups and
$m(N_1)= m(N_2)$, where $m(G)\in \mathbb{N} \cup + \infty$ (as defined in  Example \ref{examplenilpotent}) is the minimal positive integer $m$ such that 
such that for  any element $g\in G$ there exists  an integer $t >0$ such that 
such that  
$g^{m^t}=e$ (and $+\infty$ if no such $m$ exists).

\end{example}

\begin{example} Any solvable group $G$ containing at least one element of infinite order generated the same $S$ type iterated quasi-variety as $\Z$

\end{example}

We say that an $E$-type  (correspondingly $S$-type) iterated identity  $w'$ is a corollary of  an $E$-type  (correspondingly $S$-type) 
 iterated identity  $w$ if any group satisfying an $E$-type (correspondingly $S$-type) iterated identity $w$ also satisfies $w'$. Similarly, given a set of iterated identities $\Omega$ and an iterated identity
$w$ we say that $w$ is a corollary of $\Omega$ if any group satisfying all identities from $\Omega$ also satisfies $w$.

If we consider (usual) identities instead of iterated ones, this notion is not very interesting since it simply means that the 
subgroup corresponding to one of them contains the characteristic subgroup that corresponds to the other one. 
Given an iterated identity $w$,  of $E$ type or $S$ type,
  it seems interesting to describe all corollaries of $w$. Similarly, given a set $\Omega$ of iterated identities, describe the corollaries of $\Omega$.

\subsection{Open questions} \label{openquestions}

\begin{question} 
Given a group $G$, describe the  sets of iterated identities $\Omega_E(n)$ 
and $\Omega_S(n)$ that holds in this group (e.g. for Grigorchuk groups, for Golod groups, for Basilica group).

\end{question}

The same question one can ask for families of groups (finite or infinite).
For example, what can be said about the sets of iterated identities $\Omega_E(n)$ 
and $\Omega_S(n)$ 
 that hold for all solvable groups  of the group $G$)?
For examples of iterated identities $w$ in the characterizations of finite solvable groups in Remark \ref{reformulationsolvable}, what are finitely generated groups that satisfy the iterated identity $w$
(which do not need to be solvable, see Remark \ref{remarkgolod})?

\begin{question}
For which groups $G$ the characteristics subsets  $\Omega_n$ of $G$ are subgroups of $F_n$, for all $n$? 
\end{question}

\begin{question}
Describe the $E$ type and $S$ type quasi-varieties generated by a group $G$ (or by a family of groups).
\end{question}

\begin{question}
Which groups are quasi-free with respect to the iteration of identities? 
\end{question}

\begin{question}
Are free Burnside group $B_{n,p}$on $n$ generators bounded? Are free groups of a varieties of groups generated by a single identity bounded?
\end{question}

\begin{question} For which sets $\Omega \subset F_k$ there exists a group $G$ such that the set of iterated identities on $k$ variables of the group $G$ is equal to $\Omega$?

\end{question}

\begin{question}

Let $r$ be an odd prime number.
For an integer $n$ consider $w_n (x_1, x_2) = (x_1^{rn} x_2^{rn} x_1^{-rn} x_2^{-rn})^n$
Adyan shows that if we fix $n$ large enough ($n \ge 4381$ in \cite{adyan1970}    and  $n \ge 1003$   in \cite{adyanbook}[page 297] ) then the following
identities are independent , that is, none of them follows from the others.
 
 Is it true that $w_n$ , for $n$ large enough, are independent both as Engel type iterated identities and Solvability type iterated identities?

\end{question}

\end{document}